%% file: Main.tex
\documentclass[11pt]{article}

\usepackage{amsmath,amsfonts,amsthm,amssymb,color}
\usepackage{mathrsfs}
\usepackage{fancybox}
\usepackage{tikz}
\usepackage{appendix}
\usepackage{enumerate}
\usepackage{bm}
\usepackage{subeqnarray}
\usepackage{cases}
\usepackage{pifont}
\usepackage{graphicx}
\usepackage{epstopdf}
\usepackage{authblk}
\usepackage{makecell}
\usepackage{bbding}
\usepackage{subfigure}
\usepackage{threeparttable}
\usepackage{fancyhdr}
\usepackage[procnumbered,ruled,boxed,linesnumbered]{algorithm2e}
\DontPrintSemicolon
\SetKw{KwAnd}{and}
\SetProcNameSty{textsc}
\SetFuncSty{textsc}
\SetKwRepeat{Do}{do}{while}

\let\oldnl\nl
\newcommand{\nonl}{\renewcommand{\nl}{\let\nl\oldnl}}% Remove line number for one line

\ifdefined\ShowComment

\fi

\usepackage{hyperref}
\usepackage{cleveref} % must come after hyperref
\usepackage{thm-restate}
\crefformat{equation}{(#2#1#3)}

\usepackage{caption}
\usepackage{bbm}
\usepackage{tablefootnote}

\SetAlCapSkip{.5em}

\usepackage[margin=1.0in]{geometry}

\newtheorem{theorem}{Theorem}%[section]
\newtheorem{theorem1}{Theorem}

\newtheorem{theorem3}{Theorem}

\newtheorem{theorem5}{Theorem}

\newtheorem{lemma}[theorem1]{Lemma}

\newtheorem{assumption}[theorem3]{Assumption}

\theoremstyle{definition}

\newtheorem{remark}[theorem5]{Remark}

\newenvironment{fminipage}%
{\begin{Sbox}\begin{minipage}}%
		{\end{minipage}\end{Sbox}\fbox{\TheSbox}}

\input{newcommand}

\fancypagestyle{firstpage}{
    \fancyhf{}
    \chead{Mathematical Programming, In Press}
    
    \cfoot{\thepage}
}

\title{Optimal Gradient Tracking for Decentralized Optimization\thanks{Most work was done while the first author was visiting School of Data Science, The Chinese University of Hong Kong, Shenzhen. }
}

\author[1]{Zhuoqing Song}
\author[2]{Lei Shi}
\author[3]{Shi Pu\thanks{Co-corresponding authors. }}
\author[3]{Ming Yan\samethanks}
\affil[1]{Shanghai Center for Mathematical Sciences, Fudan University, Shanghai, China}
\affil[2]{School of Mathematical Sciences, Fudan University, Shanghai, China}
\affil[3]{School of Data Science, The Chinese University
of Hong Kong, Shenzhen (CUHK-Shenzhen), China}
\affil[ ]{\textit {zqsong19@fudan.edu.cn,~leishi@fudan.edu.cn,~pushi@cuhk.edu.cn,~yanming@cuhk.edu.cn}}

\date{}

\begin{document}

\maketitle

\thispagestyle{empty}

\thispagestyle{firstpage}

\input{abstract}

\section{Introduction}\label{sec:introduction1}
Consider a system of $n$ agents working collaboratively to solve the following optimization problem:
\eql{\label{eq:prob1}}{
    f\pr{\xx} = \frac{1}{n}\sum_{i=1}^n f_i\pr{\xx},
}
where $\xx\in \Real^d$ is the global decision variable, and
each $f_i\pr{\xx}$ is a smooth and strongly convex objective function accessible only to agent $i$.
The agents are connected over an undirected graph, in which they can only send/receive information directly to/from their immediate neighbors. Each agent makes its own decisions based on the decision variables and the gradients it has computed or received, that is, problem~\eqref{eq:prob1} needs to be solved in a decentralized manner under local communication and local gradient computation.

Decentralized optimization problems arise naturally in many real-world applications. For example, the data used for modern machine learning tasks are getting increasingly large, and they are usually collected or stored in a distributed fashion by a number of data centers, servers, mobile devices, etc. Centering large amounts of data is often impractical due to limited communication bandwidth and data privacy concerns~\cite{nedic2018distributed}.
Particular application scenarios of decentralized optimization  include
distributed machine learning~\cite{cohen2017projected,forrester2007multi,nedic2017fast}, wireless networks~\cite{baingana2014proximal,cohen2017distributed,mateos2012distributed}, information control~\cite{olshevsky2010efficient,ren2006consensus}, power system control~\cite{gan2012optimal,ram2009distributed}, among many others.

The study on decentralized optimization can be traced back to the early 1980s~\cite{bertsekas1983distributed,tsitsiklis1986distributed,tsitsiklis1984problems}.
Over the past decade, a large body of literature has appeared since the emergence of the distributed subgradient descent (DGD) method introduced in~\cite{nedic2009distributed}. DGD attains the optimal solution to problem~\eqref{eq:prob1} with diminishing stepsizes.
EXTRA~\cite{shi2015extra} is the first gradient-type method that achieves linear convergence for strongly convex and smooth objective functions by introducing an extra correction term to DGD.
%distributed ADMM based methods~\cite{shi2014linear,wei2012distributed} \sp{(These two are not gradient based.)}
Subsequently, distributed ADMM methods~\cite{chang2014multi,ling2015dlm}, NIDS~\cite{li2019decentralized}, and exact diffusion method~\cite{yuan2018exact} were also shown to exhibit linear convergence rates.
In particular, the convergence rates of NIDS~\cite{li2019decentralized} and EXTRA~\cite{li2020revisiting} separate the function condition number and the graph condition number and achieve the best complexities among non-accelerated methods so far \cite{xu2021distributed}.

Gradient tracking methods~\cite{di2015distributed,di2016next,nedic2017achieving,qu2017harnessing,xu2015augmented} represent another class of popular choices for decentralized optimization.
Such methods employ an auxiliary variable to track the average gradient over the entire network so that their performances become comparable to the centralized algorithms, which enjoy linear convergence.
The gradient tracking technique can be applied under uncoordinated stepsizes, directed, time-varying graphs, asynchronous settings, composite objective functions and nonconvex objective functions; see for instance~\cite{lu2021optimal,nedic2017achieving,pu2020push,saadatniaki2020decentralized,sun2019distributed,tian2020achieving,xi2017add,xin2018linear,xu2017convergence}, among many others.

In this paper, we focus on synchronous decentralized gradient-type methods for solving problem~\eqref{eq:prob1}. The algorithm efficiency is usually measured in two dimensions: (1) gradient computation complexity: the number of local gradients $\na f_i\pr{\cdot}$ that each agent $i$ needs to compute for achieving $\eps$-accuracy; (2) communication complexity: the number of communication rounds performed by each agent $i$ for achieving $\eps$-accuracy, with $O(1)$ vectors of length $O(d)$ allowed to be broadcasted to the neighbors in one communication round.
For $\mu$-strongly convex and $L$-smooth objective functions, the condition number is defined as $\ka = \frac{L}{\mu}$, which measures the ``ill-conditionedness" of the functions.
The graph condition number denoted by $\frac{1}{\gap}$ measures the ``connectivity" of the communication network, where the definition of $\gap$ is given in Section~\ref{sec:NotaAssp1}.
It was shown in~\cite{scaman2017optimal} that to obtain $\eps$-optimal solutions, the gradient computation complexity is lower bounded by $O\pr{\sqrt{\ka}\log\frac{1}{\eps}}$, and the communication complexity is lower bounded by $O\pr{\sqrt{\frac{\ka}{\gap}}\log\frac{1}{\eps}}$.

To obtain better complexities, many accelerated decentralized gradient-type methods have been developed%~\cite{dvinskikh2019decentralized,fallah2019robust,jakovetic2014fast,kovalev2020optimal,li2020decentralized,li2020revisiting,li2021accelerated,qu2019accelerated,scaman2017optimal,ye2020decentralized}
~(e.g.,~\cite{dvinskikh2021decentralized,fallah2019robust,jakovetic2014fast,kovalev2021lower,kovalev2020optimal,kovalev2021adom,li2020decentralized,li2020revisiting,li2021accelerated,qu2019accelerated,rogozin2021accelerated,rogozin2020towards,rogozin2019optimal,scaman2017optimal,uribe2020dual,xu2020accelerated,ye2020multi,ye2020decentralized}).
There exist dual-based methods such as~\cite{scaman2017optimal}
%,uribe2020dual}
that achieve optimal complexities.
	However, dual-based methods often require information related to the Fenchel duality, which is expensive to compute or even intractable in practice.
	In this paper, we focus on dual-free methods or gradient-type methods only.
Some algorithms, for instance~\cite{jakovetic2014fast,li2020decentralized,li2020revisiting,rogozin2021accelerated,rogozin2020towards,ye2020multi,ye2020decentralized}, rely on inner loops of multiple consensus steps to guarantee desirable convergence rates.
{Existing usages of inner loops can be classified into three categories: (1) directly multiplying the local decision variables with the gossip matrix for multiple times~\cite{jakovetic2014fast,rogozin2021accelerated,rogozin2020towards}; (2) Chebyshev acceleration~\cite{li2021accelerated,kovalev2020optimal,scaman2017optimal,scaman2018optimal}; (3) accelerated average consensus (FastMix)~\cite{li2020decentralized,ye2020multi,ye2020decentralized}.}
However, the inner loops place a larger communication burden~\cite{li2021accelerated,qu2019accelerated} which may limit the applications of these methods since communication has often been recognized as the major bottleneck in distributed or decentralized optimization. {For example, the optimal algorithm OPAPC~\cite{kovalev2020optimal} improves the complexities over APAPC by utilizing inner loops of Chebyshev acceleration, but its convergence is actually slower than APAPC with respect to the iteration number (communication rounds).}
In addition, inner loops impose extra coordination steps among the agents as they have to agree on when to terminate the inner loops~\cite{qu2019accelerated}.
%Besides, single-loop methods can be easily equipped with an inner loop while double-loop methods may not even be guaranteed to converge when the inner loop is unavailable in some practical scenarios.
Thus, developing single-loop methods with better complexities is of both theoretical and practical significance.

OPAPC~\cite{kovalev2020optimal} is the first decentralized gradient-type method that is optimal in both gradient computation and communication complexities.
To our best knowledge, for problem~\eqref{eq:prob1}, only OPAPC~\cite{kovalev2020optimal} and Acc-GT$+$CA~\cite{li2021accelerated}
are optimal in both complexities without additional logarithmic factors among the class of gradient-type methods.
But both methods rely on inner loops to reach the optimal complexities.
Therefore, it is natural to ask the following question:
\begin{center}
	\emph{Is there a single-loop decentralized  gradient-type method that achieves optimal gradient computation complexity and communication complexity simultaneously?}
\end{center}

Gradient tracking (GT) is one of the most popular techniques used for developing accelerated decentralized optimization methods.
Most of the existing GT-based accelerated methods rely on inner loops of multiple consensus steps or Chebyshev acceleration (CA) to reduce the consensus errors (the difference among the local variables of different agents) between consecutive steps of the outer loop; see, for instance~\cite{jakovetic2014fast,li2020decentralized,li2020revisiting,rogozin2021accelerated,rogozin2020towards,ye2020multi,ye2020decentralized}.   %\sp{(any of these optimal? better list more for comparison then less)}.
Among single-loop methods, Acc-DNGD-SC~\cite{qu2019accelerated} achieves $O\pr{\frac{\ka^{5/7}}{\gap^{3/2}}\log\frac{1}{\eps}}$ gradient computation and communication complexities.
For Acc-GT studied in~\cite{li2021accelerated}, both complexities  are $O\pr{\frac{\sqrt{\ka}}{\gap^{3/2}}\log\frac{1}{\eps}}$, and the algorithm can be applied to time-varying graphs.
Recently, APD-SC~\cite{song2021provably} was developed for general directed graphs, while the complexities are also $O\pr{\frac{\sqrt{\ka}}{\gap^{3/2}}\log\frac{1}{\eps}}$ when applied to undirected networks.
Comparing the complexities of existing GT-based methods with the lower bounds given in \cite{scaman2017optimal}, the following question arises:
\begin{center}
	\emph{Can we improve the complexities of single-loop GT-based methods to optimality?}
\end{center}

In this paper, we give affirmative answers to the two aforementioned questions with the Optimal Gradient Tracking (\OGT) method.
To our best knowledge, \OGT\ is the first single-loop algorithm that is optimal in both gradient computation and communication complexities within the class of decentralized gradient-type methods.
To develop \OGT, we propose a novel GT-based method termed $\SSGT$, which is of independent interest.
The complexity comparison of $\SSGT$ and $\OGT$ with existing state-of-the-art methods and representative GT-based methods are given in
Table~\ref{tab:SSGTOGT1}.

{\small
\renewcommand{\arraystretch}{2}
\begin{table}
  \caption{Comparison of existing state-of-the-art accelerated decentralized gradient-type methods, classical gradient tracking and some representative accelerated algorithms with $\SSGT$ and $\OGT$.}
  \label{tab:SSGTOGT1}
  \centering
  \begin{threeparttable}
  \begin{tabular}{|c|c|c|c|}
    \hline
    % after \\: \hline or \cline{col1-col2} \cline{col3-col4} ...
    \textbf{Method} & \makecell{\bf Gradient Computation \\ \textbf{Complexity}} & \makecell{\textbf{Communication} \\ \textbf{Complexity}} & \makecell{\textbf{Single-loop } }  \\ \hline
    GT~\cite{di2015distributed,di2016next,nedic2017achieving,qu2017harnessing,xu2015augmented}  & $O\pr{\pr{\ka + \frac{1}{\gap^2}}\log\frac{1}{\eps}}$  & $O\pr{\pr{\ka + \frac{1}{\gap^2}}\log\frac{1}{\eps}}$ & \Checkmark  \\ \hline
    Acc-DNGD-SC~\cite{qu2019accelerated} & $O\pr{\frac{\ka^{5/7}}{\gap^{3/2}}\log\frac{1}{\eps}}$ & $O\pr{\frac{\ka^{5/7}}{\gap^{3/2}}\log\frac{1}{\eps}}$  & \Checkmark  \\ \hline
    \makecell{Acc-GT~\cite{li2021accelerated},  APD-SC~\cite{song2021provably}  } & $O\pr{\frac{\sqrt{\ka}}{\gap^{3/2}}\log\frac{1}{\eps}}$ & $O\pr{\frac{\sqrt{\ka}}{\gap^{3/2}}\log\frac{1}{\eps}}$ & \Checkmark  \\  \hline
    \makecell{
    APAPC~\cite{kovalev2020optimal}%,
    %ADOM+~\cite{kovalev2021lower}
    }
    & $O\pr{\pr{\sqrt{\frac{\ka}{\gap}} + \frac{1}{\gap}}\log\frac{1}{\eps}}$ & $O\pr{\pr{\sqrt{\frac{\ka}{\gap}} + \frac{1}{\gap} }\log\frac{1}{\eps}}$ & \Checkmark  \\  \hline
    \makecell{
    Algorithm 3 of~\cite{kovalev2020optimal}%,
    %ADOM+~\cite{kovalev2021lower}
    }
    & $O\pr{{\sqrt{\frac{\ka}{\gap}}}\log\frac{1}{\eps}}$ & $O\pr{{\sqrt{\frac{\ka}{\gap}} }\log\frac{1}{\eps}}$ & \Checkmark  \\  \hline
    \makecell{APM-C~\cite{li2020decentralized} } & $O\pr{\sqrt{\ka}\log\frac{1}{\eps}}$ & $O\pr{\sqrt{\frac{\ka}{\gap}}\log^2\frac{1}{\eps}}$ & \xmark  \\ \hline
    \makecell{Mudag~\cite{ye2020multi}, DPAG~\cite{ye2020decentralized} } & $O\pr{\sqrt{\ka}\log\frac{1}{\eps}}$ & $O\pr{\sqrt{\frac{\ka}{\gap}}\log\ka\log\frac{1}{\eps}}$ & \xmark  \\ \hline
    \makecell{OPAPC~\cite{kovalev2020optimal},  Acc-GT+CA~\cite{li2021accelerated}
    %, \\ ADOM+ with multi-consensus~\cite{kovalev2021lower}
    } & $O\pr{\sqrt{\ka}\log\frac{1}{\eps}}$ & $O\pr{\sqrt{\frac{\ka}{\gap}}\log\frac{1}{\eps}}$ & \xmark  \\ \hline
    \bf \SSGT \ (This paper) & $O\pr{\sqrt{\ka}\log\frac{1}{\eps}}$ & $O\pr{\frac{\sqrt{\ka}}{\gap}\log\frac{1}{\eps}}$ & \Checkmark  \\ \hline
    \bf \OGT \ (This paper) &  $\bm{O\big(}\bm{\sqrt{\kappa}}\text{\textbf{log}}\frac{\bm{1}}{\bm{\epsilon}}\bm{\big)}$ &  $\bm{O\big(}\bm{\sqrt{\frac{\bm{\kappa}}{\bm{\gap}}}}\text{\textbf{log}}\frac{\bm{1}}{\bm{\epsilon}}\bm{\big)}$ & \Checkmark  \\ \hline
    \makecell{
    \textbf{Lower Bounds}~\cite{scaman2017optimal}%~\cite{scaman2017optimal,hendrikx2020optimal}\tnote{1}
    }
     & $\bm{O\big(}\bm{\sqrt{\kappa}}\text{\textbf{log}}\frac{\bm{1}}{\bm{\epsilon}}\bm{\big)}$ & $\bm{O\big(}\bm{\sqrt{\frac{\bm{\kappa}}{\bm{\gap}}}}\text{\textbf{log}}\frac{\bm{1}}{\bm{\epsilon}}\bm{\big)}$ & $\backslash$       \\ \hline
  \end{tabular}
  \end{threeparttable}

\end{table}

\subsection{Contributions  }
This paper focuses on solving the decentralized optimization problem~\eqref{eq:prob1} with first-order methods. We analyze why the previous single-loop GT-based methods have suboptimal convergence rates and first propose a novel \emph{\SSGTnm}\ (\SSGT) method.
$\SSGT$ improves upon the existing GT-based methods but does not achieve optimality in the communication complexity w.r.t. the graph condition number.
Then, we develop the \lCA\ (LCA)  technique to accelerate $\SSGT$,  which leads to the \emph{\OGTnm}\ (\OGT) method with optimal complexities.
\\~\\
The main contributions of this paper are summarized as follows:
\begin{itemize}
  \item For strongly convex and smooth objective functions, the proposed $\OGT$ method is optimal in both the gradient computation complexity and the communication complexity.
  To our knowledge, $\OGT$ is the first single-loop decentralized method that is optimal in both complexities  within the class of gradient-type methods.

  \item To develop $\OGT$, we first propose a novel gradient tracking method $\SSGT$, which is of independent interest. Compared to most existing GT-based methods that track the average gradient of the decision variables, $\SSGT$ tracks the average gradient of a ``snapshot point" instead.
  $\SSGT$ outperforms existing single-loop GT-based methods (before the development of $\OGT$) and has the potential to be extended to more general settings such as directed graphs and time-varying graphs.

  \item We propose the LCA technique, which not only accelerates the convergence of $\SSGT$, but can also be combined with many other GT-based methods and accelerate these methods with respect to the graph condition number.

  \item {We provide numerical experiments that demonstrate the superior performance of OGT compared to the state-of-the-art algorithms.}
\end{itemize}

\subsection{Roadmap}
We introduce some necessary notations and assumptions in Section~\ref{sec:NotaAssp1}.
In Section~\ref{sec:SSGT1}, we propose and analyze a novel gradient tracking method ($\SSGT$).
In Section~\ref{sec:OGT1},
the \lCA\ (LCA) technique is developed to accelerate $\SSGT$, and we study the complexities of the resulting method $\OGT$.
The efficiency of the new method is confirmed by numerical experiments in Section~\ref{sec:numerical1}.
We conclude this paper in Section~\ref{sec:conclude1}.

\section{Notations and Assumptions}\label{sec:NotaAssp1}

In this paper, $\nt{\cdot}$ denotes the Euclidean norm of vectors, and $\nf{\cdot}$ denotes the Frobenius norm of matrices.
The inner product in the Euclidean space is denoted by $\jr{\cdot, \cdot}.  $
The spectral norm of matrix $\AA$ is denoted by $\nb{\AA}$. %, where $\la_{\max}\pr{\cdot}$ is the largest eigenvalue.
We have the following relation regarding $\nf{\cdot}$ and $\nb{\cdot}$,
\eq{
	\nf{\AA\BB} \leq \nb{\AA}\nf{\BB},\ \forall \AA\in\MatSize{n}{n},\ \BB\in \MatSize{n}{d}.
}
{
By Cauchy-Schwarz inequality, for any vectors $\aa$, $\bb$ with the same length, $2\jr{\aa, \bb} \leq \frac{1 - \la}{\la}\nt{\aa}^2 + \frac{\la}{1 - \la}\nt{\bb}^2$ for any $\la \in (0, 1)$.
%\begin{lemma}
	%For any vector norm $\nm{\cdot}_*$ which can be induced by some inner product $\jr{\cdot, \cdot}_*$,
	Then, for any $\la \in (0, 1)$ and two matrices $\AA, \BB$ of the same size, we have
	\eql{\label{eq:Hilbert}}{
		\mt{\AA + \BB}^2 \leq \frac{1}{\la}\mt{\AA}^2 + \frac{1}{1 - \la}\mt{\BB}^2.
	}
	
%\end{lemma}
Using~\eqref{eq:Hilbert} recursively yields the following elementary inequality: %\sp{(replace $k$ with another letter.)}
%\begin{lemma}
	for $\la_1, \la_2, \cdots, \la_N > 0$ with $\sum_{i=1}^{N}\la_i \leq 1$, and $\AA_i\in \MatSize{m}{d}$ $(1\leq i\leq N)$,
	\eql{\label{eq:Hilbertsq}}{
		\mt{\sum_{i=1}^{N}\AA_i}^2 \leq \sum_{i=1}^{N}\frac{1}{\la_i}\mt{\AA_i}^2.
	}
	
}
	
%\end{lemma}
For a given vector $\vv$, $\vv_{a:b}$ denotes the subvector of $\vv$  containing the entries indexed from $a$ to $b$.
For a given matrix $\AA$, $\AA_{a:b, :}$ is the submatrix of $\AA$ containing the rows indexed from $a$ to $b$.
We use $\one$ to denote the all-ones vector, and $\zero$ denotes the all-zeros vector or the all-zeros matrix.
The sizes of $\one, \zero$ are determined from the context.
The operators $=, \geq, \leq$ are overloaded for vectors and matrices in the entry-wise sense.
%We denote the Frobenius norm by $\mm{}{\cdot}$.

{A sequence $\dr{a_k}_{k=0}^{\infty}$ Q-linearly converges to $0$, if there exists $c\in (0,1]$ such that  $\abs{a_{k+1}}\leq \pr{1 - c}\cdot \abs{a_k}$ for any $k\geq 0$.
A sequence $\dr{b_k}_{k=0}^{\infty}$ R-linearly converges to $0$, if there exist $C > 0$ and $c'\in (0,1]$ such that  $\abs{b_{k}}\leq C\cdot (1 - c')^k$ for any $k\geq 0$.
The concepts of Q-linear convergence and R-linear convergence will be frequently used below.
}

The agent set is denoted by $\calN = \dr{1, 2, \cdots, n}$.
{We make the following standard assumption on the local objective functions $\dr{f_i}$, which is widely used in the decentralized optimization literature; see for instance~\cite{scaman2017optimal,nedic2017achieving,xi2017add,pu2020push,xin2018linear,kovalev2020optimal}. }
\begin{assumption}\label{assp:f}
	For each $i\in \calN$, $f_i\pr{\xx}$ is $\mu$-strongly convex and $L$-smooth, i.e., for any $\xx, \yy \in \Real^{d}$,
    \eq{
        &\jr{\xx - \yy, \na f_i\pr{\xx} - \na f_i\pr{\yy}} \geq \mu \nt{\xx - \yy}^2,\\
        &\nt{\na f_i\pr{\xx} - \na f_i\pr{\yy}} \leq L\nt{\xx - \yy}.
    }
	
\end{assumption}
Therefore, the global objective function $f\pr{\xx} = \frac{1}{n}\sum_{i\in \calN} f_i\pr{\xx} $ is also $\mu$-strongly convex and $L$-smooth, and $f\pr{\xx}$ admits a unique minimizer which is denoted by $\xx^*$.
The condition number of the objective function is defined as $\ka = \frac{L}{\mu}$.

{
    The papers~\cite{ye2020decentralized,ye2020multi} only require $L$-smoothness of each $f_i$ and $\mu$-strong convexity of the average function $f = \frac{1}{n}\sum_{i=1}^{n}f_i$ rather than each $f_i$.
    However, these methods have additional logarithmic terms in the complexities and rely on inner loops of multiple consensus steps, where the iteration numbers of the inner loops depend on $\log\ka$.

    The convexity of each $f_i$ is essential in the proofs of this paper as well as in the paper~\cite{scaman2017optimal}, which gives the lower bounds, and the methods APAPC, OPAPC in~\cite{kovalev2020optimal} and Acc-GT, Acc-GT+CA in~\cite{li2021accelerated}, which are most closely related to this paper.
}

In decentralized optimization algorithms, each agent employs some local $d$-dimension row vectors represented by bold lower-case letters.
We use subscripts to denote the owners of the variables and superscripts to indicate the iteration number.\footnote{Except
%$\WW^k = \underbrace{\WW\cdot \cdots \cdot \WW}_{\text{multiplication of $\WW$ for $k$ times}}$
in $\WW^k$, $\Wt^k$ where $k$ is the power number, all the superscripts $k+1, k, k-1, \cdots $ in this paper refer to the iteration number.  }
For instance, $\xx^k_i$ denotes the value of agent $i$'s local variable $\xx$ at iteration $k$.
The local vectors are usually written compactly into $n$-by-$d$ matrices with the same letter but in bold upper-case style.
For instance,
\eq{
	\XX^k =
	\begin{pmatrix}
		\frac{ \quad }{ } & \xx^k_1 & \frac{ \quad }{ }  \\
		\frac{ \quad }{ } & \xx^k_2 & \frac{ \quad }{ }  \\
		& \vdots &  \\
		\frac{ \quad }{ } & \xx^k_n & \frac{ \quad     }{ }
	\end{pmatrix}.
	%\pr{\xx^k_1, \xx^k_2, \cdots, \xx^k_n    } \in \MatSize{n}{d}.
}
The local gradients $\dr{\na f_i\pr{\xx^k_i}}_{i\in \calN}$ at iteration $k$ are also written compactly in an $n$-by-$d$ matrix as follows %For instance, %\sp{(Define $F$ first.)}
\eq{
	\gF{\XX^k} =
	\begin{pmatrix}
		\frac{ \quad }{ } & \na f_1\pr{\xx^k_1} & \frac{ \quad }{ }  \\
		\frac{ \quad }{ } & \na f_2\pr{\xx^k_2} & \frac{ \quad }{ }  \\
		& \vdots &  \\
		\frac{ \quad }{ } & \na f_n\pr{\xx^k_n} & \frac{ \quad     }{ }
	\end{pmatrix}.
	%\pr{\na f_1\pr{\xx^k_1}, \na f_2\pr{\xx^k_2}, \cdots, \na f_n\pr{\xx^k_n}    }.
}

For any $n$-by-$d$ matrix, we overline the same letter in bold lower-case style to denote the average of its rows, e.g.,
\eq{
	\xa^k = \frac{1}{n}\one\tp\XX^k =  \frac{1}{n}\sum_{i\in\calN}\xx^k_i.
}

The agents are connected through a graph $\calG = \pr{\calN, \calE}$ with $\calE \subset \calN \times \calN$ being the edge set. The graph $\calG$ is assumed to be undirected and connected.
The information exchange between the agents is realized through a gossip matrix $\WW$, which satisfies the following standard condition. %\sp{(Change $k$ to another letter)}
\begin{assumption}\label{assp:W}
{
    The gossip matrix $\WW$ satisfies $\WW\one = \one$, $\WW\tp\one = \one$, and there exists a constant $\gap\in (0, 1]$ such that
    %\begin{fact}
    \eql{\label{eq:nbConW1}}{
    	\nb{\WW - \frac{\one\one\tp}{n}} = 1 - \gap.
    }
    %\end{fact}

}
	
\end{assumption}
\begin{remark}\label{rem:asspWasym1}
    {
        We can instantiate a matrix $\WW$ satisfying Assumption~\ref{assp:W} easily by choosing a doubly stochastic $\WW$, which additionally satisfies $\WW_{ij} > 0$ for any $(i,j)\in \calE$ and $\WW_{ii} > 0 $ for any $i\in \calN$.
    }
\end{remark}

{
Assumption~2 is also standard in decentralized optimization; see for instance, DGD based methods such as~\cite{nedic2009distributed,johansson2010randomized,koloskova2020unified,yuan2016convergence} and gradient tracking methods such as~\cite{qu2017harnessing,di2015distributed,di2016next,qu2019accelerated,pu2020distributed}, among many others.
Since Assumption~2 does not require symmetricity, it is slightly weaker than the graph conditions in most primal-dual methods such as~\cite{shi2015extra,scaman2017optimal,scaman2018optimal,li2019decentralized,kovalev2020optimal}.
}

{
    The convergence of SS-GT is proved under Assumption~\ref{assp:f} and Assumption~\ref{assp:W}.
    The analysis for the OGT method requires another standard assumption stated below.
}
%We make the following additional assumption on the gossip matrix.
\begin{assumption}\label{assp:Wsym}
	(For OGT only)
    The gossip matrix $\WW$ is symmetric and positive semidefinite.
\end{assumption}
\begin{remark}
    To meet Assumption~\ref{assp:W} and~\ref{assp:Wsym} simultaneously, we can simply choose a symmetric doubly stochastic matrix $\widehat{\WW}$ satisfying the conditions in Remark~\ref{rem:asspWasym1}.
    And then, set $\WW = \pr{\II + \widehat{\WW}}/2  $.
\end{remark}

{
    Assumptions~\ref{assp:f},~\ref{assp:W},~\ref{assp:Wsym} are all standard assumptions in decentralized optimization and have been used in many works together; see for instance~\cite{shi2015extra,scaman2017optimal,li2019decentralized,kovalev2020optimal}.
    Particularly, the lower bounds given in~\cite{scaman2017optimal} and the existing optimal methods OPAPC~\cite{kovalev2020optimal} and Acc-GT+CA~\cite{li2021accelerated} are all analyzed under Assumptions~\ref{assp:f},~\ref{assp:W},~\ref{assp:Wsym}.
}

If $\WW$ is symmetric,
from the spectral decomposition, the number $\gap$ defined in~\eqref{eq:nbConW1} is indeed the spectral gap of $\WW$, which refers to the difference between the moduli of the two largest eigenvalues.
And the number $\frac{1}{\gap}$ is considered as the condition number of the communication network in decentralized optimization.

Define the projection matrix
\eq{
	\Con = \II - \frac{\one\one\tp}{n}.
}
It follows directly that $\Con\WW = \Con\WW\Con$ and $\nb{\Con} = 1$.

\def\Pr#1{\textbf{Pr}\br{#1}}
\section{\SSGTnm\     }\label{sec:SSGT1}

In this section, we propose a novel gradient tracking method named \emph{\SSGTnm}\ (SS-GT), which will serve as a building block for $\OGT$ in Section~\ref{sec:OGT1}.
We start with the motivation and explain how \SSGT\ is constructed in Section~\ref{sec:SSGT_motivation}.
In Section~\ref{sec:DOGTconsenLya1}, we construct a Lyapunov function to bound the consensus errors.
In Section~\ref{sec:SSGTrate}, we demonstrate the gradient computation and communication complexities of $\SSGT$.

%To help analyze the performance of $\SSGT$, we rewrite Algorithm~\ref{alg:DiOGT} into a more compact form:
The $\SSGT$ method starts with the initial values:
\eql{\label{eq:initQ0}}{
    \YY^0 = \ZZ^0 = \UU^0 = \QQ^0 = \XX^0 ,\ \GG^0 = \gF{\QQ^0}
}
and updates, for $k=0,1,2, \dots$,
\subeqnuml{\label{eq:SnapshotGT}}{
    %\text{Sample $\xi^k\sim Bernoulli(p)$ and $\zeta^k\sim Bernoulli (q)$}  \\
	\XX^{k} = \pr{1 - \alp- \tau}\YY^k + \alp\ZZ^k + \tau\UU^k \label{eq:X}  \\
	\ZZ^{k+1} = \pr{1 + \bet}\inv\WW\Big( \ZZ^k + \bet\XX^k \notag \\
    \qquad  - \eta\pr{\GG^k + {\zeta^k}\pr{ \gF{\XX^k}-\gF{\QQ^k} }}\Big) \label{eq:Z}  \\
	\YY^{k+1} = \XX^k + \gam\pr{\ZZ^{k+1} - \ZZ^k} \label{eq:Y} \\
	\QQ^{k+1} = \xi^{k}\XX^k + \pr{1 - \xi^k}\QQ^k \label{eq:Q}  \\
	\UU^{k+1} = \WW\pr{\xi^{k}\XX^k + \pr{1 - \xi^k}\UU^k}  \label{eq:U} \\
	\GG^{k+1} = \WW\GG^k + \xi^k\pr{\gF{\XX^k} - \gF{\QQ^k}} \label{eq:G}
	%, \XX^{k+1} = \pr{1 - \alp- \tau}\YY^{k+1} + \alp\ZZ^{k+1} + \tau\UU^{k+1} \label{eq:X}
}
where $\dr{\xi^k, \zeta^k}$ are two-point random variables with $\xi^k\sim Bernoulli\pr{p}$ and $\zeta^k\sim Bernoulli\pr{q}/q$.\footnote{The notation $\zeta^k\sim Bernoulli\pr{q}/q$ means $\Pr{\zeta^k=\frac{1}{q}} = q$ and $\Pr{\zeta^k = 0} = 1 - q$. }

An implementation-friendly version of $\SSGT$ is given in Appendix~\ref{sec:ifDOGTandUOGT} (see Algorithm~\ref{alg:DiOGT}). The equivalence between~\eqref{eq:SnapshotGT} and Algorithm~\ref{alg:DiOGT} comes from the observation that $\MM^k$ in Algorithm~\ref{alg:DiOGT} equals $\gF{\QQ^k}$ in~\eqref{eq:SnapshotGT}.

\begin{remark}
	At each iteration, $\xi^k$ and $\zeta^k$ are two parameters shared with all agents, and all agents need to generate the same random numbers $\xi^k, \zeta^k$ locally.
	%	However, we require that $\xi^k, \zeta^k$ of different agents are the same.
	This can be done by letting the agents share a common random seed at the beginning and use the same random number generator initialized by this common random seed to generate $\dr{\xi^k, \zeta^k}_{k\geq 0}$.
	A similar technique was used  in~\cite{scaman2018optimal} to let the agents generate the same sequence of Gaussian random variables without communication.
\end{remark}

\begin{remark}
	The agents need to compute the gradient of $\XX^k$ only when $\xi^k \neq 0$ or $\zeta^k \neq 0 $.
	Hence if $p, q$ are small, in most iterations, the agents only fuse the information received from their neighbors without computing gradients.
	If $\xi^k$ and $\zeta^k$ are independent, $p + (1- p)q$ gradient computations are required in expectation for each iteration.
	However, by setting $p = q$ and $\xi^k = q\zeta^k$, only $p$ gradient computations are required in expectation for each iteration.
\end{remark}

\subsection{Motivation and preliminary analysis}\label{sec:SSGT_motivation}
In this part, we first review the classical gradient tracking methods and their accelerated versions.
By investigating why these methods are suboptimal, we provide the motivation behind \SSGT.
Then, we describe the roadmap to prove the complexities of \SSGT\ and provide some preliminary analysis.

\subsubsection{Related gradient tracking based methods}
The classical gradient tracking (GT) methods achieve linear convergence for smooth and strongly-convex objective functions. They can be implemented with uncoordinated stepsizes and generalized to using row and/or column stochastic mixing matrices.
With the initialization $\SS^0 = \gF{\XX^0}$, the simplest GT implementation works as follows:
\eq{
	\XX^{k+1} &= \WW\XX^k - \eta\SS^k,  \\
	\SS^{k+1} &= \WW\SS^k + \gF{\XX^{k+1}} - \gF{\XX^k},
}
where $\XX^k\in\mathbb{R}^{n\times d}$ contains the local decision variables and $\SS^k\in \MatSize{n}{d}$ is called the ``gradient tracker".
A typical way to analyze the performance of GT starts from the following decomposition:
\eq{
	\XX^k &= \XX^k - \frac{1}{n}\one\one\tp\XX^k + \frac{1}{n}\one\one\tp\XX^k = \Con\XX^k + \one\xa^k,\\
	\SS^k &= \SS^k - \frac{1}{n}\one\one\tp\SS^k + \frac{1}{n}\one\one\tp\SS^k = \Con\SS^k + \one\agt^k,
}
where the notations $\Con, \xa^k$, and $\agt^k$ have been introduced in Section~\ref{sec:NotaAssp1}.
%The terms $\mt{\Con\XX^k}^2, \mt{\Con\YY^k}^2$ are called ``consensus errors".
The ``consensus error" $\mt{\Con\XX^k}^2 = \sum_{i\in\calN}\nt{\xx^k_i - \xa^k}^2$ is the summation of the squared distances between each local decision variable to their average.
In light of property~\eqref{eq:nbConW1}, we have $\nf{\Con\WW\XX^k} = \nf{\Con\WW\Con\XX^k} \leq \nb{\Con\WW}\nf{\Con\XX^k} = \pr{1 - \gap}\nf{\Con\XX^k}$.
This indicates that multiplying $\XX^k$ by $\WW$ reduces the consensus errors.

%To understand why $\YY^k$ is the gradient tracker,
By Assumption~\ref{assp:W}, we have $\one\tp\WW = \one\tp$, and therefore,
\eq{
	\agt^k =& \frac{1}{n}\one\tp\pr{\WW\SS^{k-1} + \gF{\XX^k} - \gF{\XX^{k-1}}} \\
    =& \agt^{k-1} + \frac{\one\tp\pr{\gF{\XX^k} - \gF{\XX^{k-1}}}}{n}.
}
%where the notation $\ya^k = \frac{1}{n}\one\tp\YY^k$ is the average of the rows of $\YY^k$ as we have defined in Section~\ref{sec:NotaAssp1}.
It follows by induction that
\eq{
	\agt^k = \frac{1}{n}\one\tp\gF{\XX^k}.
}
Note that  $\frac{1}{n}\one\tp\gF{\one\xa^k} = \na f\pr{\xa^k}$.
When the consensus error of $\XX^k$ is small, $\agt^k$ will be close to the true gradient $\na f\pr{\xa^k}$.

Based on the above observations, it is important to answer the following questions for analyzing and accelerating the convergence of GT-based methods.
\emph{
	\begin{enumerate}[\textbf{Question}~I:]%\setcounter{enumi}{0}
		\item  How to bound the errors coming from inexact gradients $\agt^k \neq \na f\pr{\xa^k}$? \label{que:I}
		\item  How to decrease the consensus error of $\XX^k$ fast?  \label{que:II}
	\end{enumerate}
}
\que{I} can be answered by Lemma~\ref{lem:inextgt1}, which is often used in the analysis of GT-based methods.
Lemma~\ref{lem:inextgt1} bounds the errors induced by inexact gradients through consensus errors,
which reduces \que{I} to \que{II}.
\begin{lemma}\label{lem:inextgt1}
	Denote $\da^k = \frac{1}{n}\gF{\XX^k}$.
	For any row vectors $\aa, \bb \in \Real^d  $ and $k \geq 0$,
	\eql{\label{eq:inextL}}{
		f\pr{\aa} \leq     f\pr{\bb} + \jr{\da^k, \aa - \bb} + L\nt{\xa^k - \aa}^2 + \frac{L}{n}\mt{\Con\XX^k}^2,
	}
	and
	\eql{\label{eq:inextmu}}{
		f\pr{\xa^k} \leq f\pr{\xx^*} + \jr{\da^k, \xa^k - \xx^*} - \frac{\mu}{4}\nt{\xa^k - \xx^*}^2 + \frac{L}{n}\mt{\Con\XX^k}^2.
	}
\end{lemma}
The proof of Lemma~\ref{lem:inextgt1} can be found in, for instance~\cite{jakovetic2014fast,qu2017harnessing,qu2019accelerated}, and it is included in Appendix~\ref{sec:prfleminextgt1} for completeness.

To address \que{II}, probably the simplest way is to add inner loops of multiple consensus steps in the algorithm.
However, inner loops have several major drawbacks compared with single-loop methods, as discussed in the introduction.
Without inner loops, controlling the consensus errors requires sufficiently small stepsizes which usually result in slow convergence.
Many efforts have been made to overcome this issue.

The acc-DNGD-SC method in~\cite{qu2019accelerated} combines Nesterov's accelerated gradient descent with GT directly in the following way:
\eq{
	\YY^0 &= \ZZ^0 = \XX^0,\ \SS^0 = \gF{\XX^0},  \\
	\ZZ^{k+1} &= \pr{1 - \alp}\WW\ZZ^k + \alp\WW\XX^k - \frac{\eta}{\alp}\SS^k,  \\
	\YY^{k+1} &= \WW\XX^k - \eta\SS^k,  \\
	\XX^{k+1} &= \frac{\YY^{k+1} + \alp\ZZ^{k+1}}{1 + \alp},  \\
	\SS^{k+1} &= \WW\SS^k + \gF{\XX^{k+1}} - \gF{\XX^k},
}
where $\SS^k$ is the gradient tracker.
The communication complexity and gradient computation complexity of acc-DNGD-SC are both $O\pr{\frac{\ka^{5/7  }}{\gap^{3/2}}\log\frac{1}{\eps}}$.

\newcommand\thek{\beta}
A recent work~\cite{li2021accelerated} considered accelerated GT over time-varying graphs.
When applied to static graphs, the proposed Acc-GT method works as follows:
\eq{
	\YY^0 &= \ZZ^0 = \XX^0,\  % = \one\vv^0 \ (\vv^0\in \Real^d),\ \
    \SS^0 = \gF{\XX^0},   \\
	\ZZ^{k+1} &= \frac{1}{1 + \frac{\mu\alp}{\thek}}\pr{\WW\pr{\frac{\mu\alp}{\thek}\XX^k + \ZZ^k} - \frac{\alp}{\thek}\SS^k},  \\
	\YY^{k+1} &= \thek\ZZ^{k+1} + \pr{1 - \thek}\WW\YY^k,  \\
	\XX^{k+1} &=  \thek\ZZ^{k+1} + \pr{1 - \thek}\YY^{k+1},  \\
	\SS^{k+1} &= \WW\SS^{k} + \gF{\XX^{k+1}} - \gF{\XX^{k}},
}
where $\SS^k$ is the gradient tracker.
Acc-GT was proven to have communication complexity and gradient computation complexity $O\pr{\frac{\sqrt{\ka}}{\gap^{3/2}}\log\frac{1}{\eps}}$, which is optimal in the function condition number $\ka$ but suboptimal in $\frac{1}{\gap}$.
Such a result implies that better answers for \que{II} are imperative for improving the algorithmic complexities. In particular, notice that the aforementioned methods perform one gradient computation and $O(1)$ communication steps in one iteration. Hence they always achieve the same gradient computation complexity and communication complexity.
However, the  lower bound on the gradient computation complexity $O\pr{\sqrt{\ka}\log\frac{1}{\eps}}$ obtained in~\cite{scaman2017optimal} is independent of $\gap$. %{\color{red}a citation is needed.}
To develop optimal methods, we must address the following question.
\emph{
	\begin{enumerate}[\textbf{Question}~I:]\setcounter{enumi}{2}
		\item  How to get rid of $\gap$ in the gradient computation complexity?  \label{que:III}
	\end{enumerate}
}

\subsubsection{Motivation for \SSGT\ and preliminary analysis}
The aforementioned GT-based methods update the gradient tracker in the following way:
\eq{
	\SS^{k+1} = \WW\SS^k + \gF{\XX^{k+1}} - \gF{\XX^k}.
}
In accelerated methods, due to the Nesterov momentum, the distance between $\XX^{k+1}$ and $\XX^k$ can be large, which implies that $\gF{\XX^{k+1}} - \gF{\XX^k}$ can also be large. This further leads to large consensus errors for the gradient tracker. As a result,  small stepsize has to be used to control the consensus errors, and it eventually leads to suboptimal convergence rates.
%Therefore, we are motivated to design new gradient trackers with less consensus errors.

%To reduce the consensus error in the gradient tracker,
To obtain a better answer for \que{II},
inspired by SVRG~\cite{johnson2013accelerating}, Katyusha~\cite{allen2017katyusha}, Loopless SVRG~\cite{kovalev2020don}
and ADIANA~\cite{li2020acceleration},
we introduce a ``snapshot point" $\QQ^k$, which records some history positions of $\XX^k$.
Unlike previous GT-based methods whose gradient tracker $\SS^k$ tracks the average gradient of $\XX^k$,
a new gradient tracker $\GG^k$ is employed by $\SSGT$ to track the average of $\gF{\QQ^k}$.
This is shown by the following lemma.
%Using the random variable $\xi^k$ also helps addressing \que{III}. When $p$ is small, e.g., $O(\gap)$ in $\SSGT$ or $O(\sqrt{\gap})$ in $\OGT$, the gradient computation complexity can be smaller than the communication complexity.
\begin{lemma}\label{lem:gakm}
	For any $0 \leq k \leq K$,
	\eql{\label{eq:gakm}}{
		\ga^k = \frac{1}{n}\one\tp\gF{\QQ^k}.
	}
	
\end{lemma}
\begin{proof}
	See Appendix~\ref{sec:prfgakm1}.
	\end{proof}

{
    To avoid an outer loop for updating the ``snapshot point" $\QQ^k$ so that SS-GT can be a single-loop algorithm, we consider the technique used in Loopless SVRG~\cite{kovalev2020don} where ``snapshot" points are updated in a coin-flip manner.
    More specifically, the update of ``snapshot point" $\QQ^k$ follows the Bernoulli distribution with parameter $p$.
}
 When $p$ is small, $\QQ^k$ is not updated frequently, and when $\QQ^k$ is not updated, we multiply the gradient tracker $\GG^k$ with $\WW$ to reduce its consensus error.
%The following lemma shows that the average of $\QQ^k$ is the same as the average of $\gF{\QQ^k}$.
When $\QQ^k$ is updated, the gradient different $\gF{\XX^k} - \gF{\QQ^k}$ is added into the gradient tracker $\GG^k$ (see~\eqref{eq:G}).
To avoid the difference $\gF{\XX^k} - \gF{\QQ^k}$ in~\eqref{eq:Z} and~\eqref{eq:G} becoming too large, we add a ``negative momentum" in the linear coupling step~\eqref{eq:X}, which is inspired from Katyusha~\cite{allen2017katyusha}.
The ``negative momentum" behaves like a ``magnet" that retracts $\XX^k$ towards some historical positions.
A straightforward choice for the history position is the ``snapshot point" $\QQ^k$.
In this way, we can avoid $\XX^k$ from getting too far away from $\QQ^k$, and $\gF{\XX^k} - \gF{\QQ^k}$ won't be too large.
However, the consensus error in $\QQ^k$ is not reduced during the iteration because it does not change if $\xi^k=0$. So, it does not help to reduce the consensus error in $\XX^k$.
We introduce another variable $\UU^k$ with a smaller consensus error than $\QQ^k$ as the historical position. When $\QQ^k$ is updated using the recent $\XX^{k-1}$, we also updated $\UU^k$ using $\WW\XX^{k-1}$. Whenever $\QQ^k$ is not updated, we update $\UU^k$ by multiplying it with $\WW$. Then, it can be proved easily by induction that   $\UU^k$ has the same average as $\QQ^k$, i.e.,
 \eql{\label{eq:efctua1}}{
 	\ua^k = \qa^k,\ \forall 0 \leq k\leq K.
}but $\UU^k$ has a smaller consensus error than $\QQ^k$.

The above discussion intuitively explains how we address \que{II}.

The new gradient tracker $\GG^k$ has an additional advantage.
To compute the gradient tracker $\SS^{k}$, the agents needs to compute $\gF{\XX^{k}}$.
Therefore, the gradient computation complexities of most existing GT-based methods are always no less than their communication complexities.
However, for our new $\GG^k$, when $\QQ^k$ is not updated (i.e., $\xi^k = 0$), there is no need to compute the gradients   in step~\eqref{eq:G}.

Nevertheless, the new gradient   tracker $\GG^k$ also has a drawback.
The previous gradient tracker $\SS^k$ satisfies $\agt^k = \frac{1}{n}\gF{\XX^k}$, which indicates that $\agt$ serves as a good estimator for $\na f\pr{\xa^k}$ when the consensus error is small. However, due to the distance between $\XX^k$ and $\QQ^k$, $\ga^k$ may not be an ``accurate enough" gradient estimator for $\na f\pr{\xa^k}$.
A naive way is to use
\eq{
    \DD^k =  \GG^k + \gF{\XX^k} - \gF{\QQ^k}
}
as a gradient estimator instead because we can show by~\eqref{eq:gakm} that $\da^k = \frac{1}{n}\one\tp\DD^k =  \frac{1}{n}\one\tp\gF{\XX^k}$.
However, computing a new $\DD^k$ in each iteration requires at least one gradient computation per iteration.
This prevents us from solving \que{III}.
Inspired by SPIDER~\cite{fang2018spider},
we introduce the random variable $\zeta^k$ and use the term $\GG^k + \zeta^k\pr{\gF{\XX^k} - \gF{\QQ^k}}$ in~\eqref{eq:Z} which is an unbiased estimator for $\one\tp\gF{\XX^k}$ as
\eq{
	\frac{1}{n}\one\tp\E{\GG^k + \zeta^k\pr{\gF{\XX^k} - \gF{\QQ^k}}} = \frac{1}{n}\one\tp\gF{\XX^k},
}
where the equality comes from $\E{\zeta^k} = 1$ and~\eqref{eq:gakm}.
Since we need to compute the gradients only when $\xi^k \neq 0$ or $\zeta^k \neq 0$, by setting $p$, $q$ to be in the order of $O\pr{\gap}$ in \SSGT\ or $O\pr{\sqrt{\gap}}$ in \OGT, the gradient computation complexity is shown to be independent of $\gap$ for both methods.
At the same time, the communication complexities are not affected.
This explains how we address~\que{III}.

Hereafter, we will denote
\eq{
	\da^k = \frac{1}{n}\one\tp\gF{\XX^k}.
}
The $n$-by-$d$ matrices such as $\XX^k$ will be decomposed as its consensus error and the average part (here, we take $\XX^k$ as an example):
\eq{
	\XX^k = \Con\XX^k + \one\xa^k.
}
Multiplying $\frac{\one\tp}{n}$ on both sides of~\eqref{eq:X}-\eqref{eq:G} and combining with~\eqref{eq:gakm} and~\eqref{eq:efctua1} yield
\subeqnuml{\label{eq:aupdt1}}{
	\xa^{k} = \pr{1 - \alp - \tau}\ya^k + \alp\za^k + \tau\ua^k, \label{eq:xa}  \\
	\za^{k+1} = \pr{1 + \bet}\inv\pr{\za^k + \bet\xa^k - \eta\ga^k + \zeta^k\eta\pr{\ga^k - \da^k}}, \label{eq:za}  \\
	\ya^{k+1} %= \pr{1 - \alp - \tau} + \gam\za^{k+1} + \tau\ua^k
	= \xa^k + \gam\pr{\za^{k+1} - \za^k}, \label{eq:ya} \\
	\qa^{k+1} = \ua^{k+1} = \xi^{k}\xa^k + \pr{1 - \xi^k}\ua^k,  \label{eq:ua}  \\
	\ga^{k+1} = \frac{1}{n}\one\tp\gF{\QQ^{k+1}}. \label{eq:ga}
}

In the rest of this section, we will first construct a Lyapunov function to bound the consensus errors by a ``Q-linear" sequence with additional errors in the form of gradient differences.
Then, we provide bounds for the descent of the average parts in~\eqref{eq:aupdt1} with additional errors in terms of consensus errors.
Finally, we combine both parts to show the complexities of $\SSGT$.

\subsection{Bounding the consensus errors  }\label{sec:DOGTconsenLya1}
In this section, we construct a Lyapunov function to bound the consensus errors of $\XX^k$ and $\QQ^k$.
To see why this is necessary, first note that later in Section~\ref{sec:SSGTrate}, we will invoke Lemma~\ref{lem:inextgt1} to bound the descent of the average parts in~\eqref{eq:aupdt1}.
Since Lemma~\ref{lem:inextgt1} pertains to the errors induced by inexact gradients in terms of $\mt{\Con\XX^k}^2$,  the consensus error of $\XX^k$ needs to be bounded.
Second, there are certain errors in terms of $\mt{\gF{\XX^k} - \gF{\QQ^k}}^2$ on the RHS of~\eqref{eq:dajrz+1} in Lemma~\ref{lem:dajrsupp}.
By~\eqref{eq:diffgF} in Lemma~\ref{lem:diffgF12}, the gradient difference $\mt{\gF{\XX^k} - \gF{\QQ^k}}^2$ is related to $\mt{\Con\XX^k}^2$ and  $\mt{\Con\QQ^k}^2$,
and hence  $\mt{\Con\QQ^k}^2$ also needs to be bounded.
The consensus errors of the other variables, including $\YY^k$, $\ZZ^k$, $\GG^k$, $\UU^k$ will only occur in the Lyapunov function in Lemma~\ref{lem:conbd1} and not in the main analysis of $\SSGT$. Thus we can regard $\mt{\Con\YY^k}^2$, $\mt{\Con\ZZ^k}^2$, $\mt{\Con\GG^k}^2$, $\mt{\Con\UU^k}^2$ as ``auxiliary variables" which help bound the consensus errors of $\XX^k$ and $\QQ^k$.

Regarding the magnitudes of the related quantities,
we mention in advance that in \SSGT\
we will choose
\eql{\label{eq:guideparastp1}}{
	\text{
	%$\tau = \frac{1}{2}$,
	$p, q = O\pr{\gap}$, $\alp, \gam = O\pr{\frac{1}{\sqrt{\ka}}}$, $\eta = O\pr{\frac{\gap\sqrt{\ka}}{L}}$, $\bet = O\pr{\frac{\gap}{\sqrt{\ka}}}$,}
}
and $\tau$ is chosen as a constant in $(0, 1)$ such as $\tau = \frac{1}{2}$.
In light of the magnitudes mentioned in~\eqref{eq:guideparastp1}, inequalities~\eqref{eq:Lyparastp} can be satisfied easily by choosing the parameters properly.

Let $\calF_k$ denote the sigma field generated by $\dr{\xi^j, \zeta^j}_{0\leq j \leq k-1}$, and $\Ek{ \ \cdot \ } $ is the conditional expectation taken over $\calF_k$.
Then, for any $k\geq 0$, the variables $\{\XX^k, \YY^k, \ZZ^k, \QQ^k, \UU^k, \GG^k \}  $ are measurable with respect to $\calF_k$.
%We remark that to help readers see the proofs more clearly, we leave some common factors on both sides of the inequalities which serve as requirements for some lemmas such as~\eqref{eq:Lyparastp},~\eqref{eq:parastpdiffgF}.

\begin{lemma}\label{lem:conbd1}
	If the parameters and stepsize satisfy\footnote{To help readers follow the proofs more easily, we leave some common factors on both sides of the inequalities such as~\eqref{eq:Lyparastp},~\eqref{eq:parastpdiffgF},~\eqref{eq:thmsnapshotGTcond1} which serve as requirements for some lemmas.}
	\subeql{\label{eq:Lyparastp}}{
		&\frac{2\pr{1 - \tau}}{2 - \tau} + \frac{192\pr{1 - \tau}\gam^2\bet^2}{\tau\gap^2} + \frac{\tau}{4} + \frac{16 p}{\gap}  \leq 1 - \frac{\tau}{8},  \\
		%&\frac{17\tau  \gap}{40} \leq \frac{3\gap}{2} - \gap^2  \\
		&\frac{\alp^2}{\pr{1 - \tau}\gam^2} \leq 1,  \\
		&\frac{192\pr{1 - \tau}\gam^2\eta^2  L^2}{\tau\gap} \leq \frac{\gap}{2},
	}
	then
	\eql{\label{eq:LyConX}}{
		\mt{\Con\XX^k}^2 + \mt{\Con\QQ^k}^2 \leq& \pr{1 - \min\dr{\frac{p}{2}, \tau, \frac{\gap}{2}}}\Ly^k - \Ek{\Ly^{k+1}}  \\
		& + \pb\mt{\gF{\XX^k} - \gF{\QQ^k}}^2,
	}
	where the Lyapunov function is given by
	\eq{
		\Ly^k = \frac{8}{\tau}{\Big(}\frac{\tau}{4p}\mt{\Con\QQ^k}^2 + &\frac{4\pr{1 - \tau}}{4 - \tau}\mt{\Con\YY^{k}}^2 + \frac{16}{\gap}\mt{\Con\UU^{k}}^2  \\
		+ &\pr{1 - \frac{\gap}{6}}\frac{48\pr{1 - \tau}\gam^2}{\tau\gap}\mt{\Con\ZZ^{k}}^2
		+ \frac{1}{\gap  L^2}\mt{\Con\GG^{k}}^2{\Big)},
	}
	and \eql{\label{eq:defpb1}}{\pb = \frac{384\pr{1 - \tau}\gam^2}{\tau^2\gap}\pr{\frac{\eta^2}{q} + \frac{2\eta^2}{\gap}} + \frac{16  p}{\tau  \gap  L^2}.    }
	
\end{lemma}
\begin{proof}
	By multiplying $\Con$ on both sides of~\eqref{eq:X}, and using~\eqref{eq:Hilbertsq},
	we have
	\eql{\label{eq:ConX}}{
		\mt{\Con\XX^k}^2 &\leq \frac{4}{4 - \tau  }\mt{\pr{1 - \alp - \tau}\Con\YY^k}^2 + \frac{8}{\tau}\mt{\alp\Con\ZZ^k}^2 + \frac{8}{\tau}\mt{\tau\Con\UU^k}^2  \\
		&\leq \frac{4\pr{1 - \tau}^2}{4 - \tau}\mt{\Con\YY^k}^2 + \frac{8\alp^2}{\tau}\mt{\Con\ZZ^k}^2 + 8\tau\mt{\Con\UU^k}^2.
	}
	
	To bound the consensus error of $\ZZ^{k+1}$, we first analyze the following term:
	\eql{\label{eq:ConYsp1}}{
		&\Ek{\mt{\Con\ZZ^k + \bet\Con\XX^k - \eta\Con\GG^k + \zeta^k\eta\Con\pr{\gF{\QQ^k} - \gF{\XX^k}}}^2}  \\
		=& \mt{\Con\ZZ^k + \bet\Con\XX^k - \eta\Con\GG^k}^2 + \Ek{\pr{\zeta^k}^2}\eta^2\mt{\Con\pr{\gF{\QQ^k} - \gF{\XX^k}}}^2  \\
		& + 2\jr{\Con\ZZ^k + \bet\Con\XX^k - \eta\Con\GG^k, \eta\Ek{\zeta^k}\Con\pr{\gF{\QQ^k} - \gF{\XX^k}}}  \\
		\leq& \mt{\Con\ZZ^k + \bet\Con\XX^k - \eta\Con\GG^k}^2 + \frac{\eta^2}{q}\nb{\Con}^2\mt{\gF{\QQ^k} - \gF{\XX^k}}^2  \\
		& + \frac{\gap}{2}\mt{\Con\ZZ^k + \bet\Con\XX^k - \eta\Con\GG^k}^2 + \frac{2\eta^2}{\gap}\nb{\Con}^2\mt{\gF{\QQ^k} - \gF{\XX^k}}^2  \\
		\comleq{\eqref{eq:Hilbertsq}}& \pr{1 + \frac{\gap}{2}}\pr{\frac{1}{1 - \frac{\gap}{2}}\mt{\Con\ZZ^k}^2 + \frac{4\bet^2}{\gap}\mt{\Con\XX^k}^2 + \frac{4\eta^2}{\gap}\mt{\Con\GG^k}^2}  \\
		& + \pr{\frac{\eta^2}{q} + \frac{2\eta^2}{\gap}}\mt{\gF{\QQ^k} - \gF{\XX^k}}^2.
	}
	
	By multiplying $\Con$ on both sides of~\eqref{eq:Z}, we have
	\eql{\label{eq:ConZ}}{
		&\Ek{\mt{\Con\ZZ^{k+1}}^2}  \\
        \leq& \pr{1 + \bet}^{-2}\nb{\Con\WW}^2\Ek{\mt{\Con\ZZ^k + \bet\Con\XX^k - \eta\Con\GG^k + \zeta^k\eta\Con\pr{\gF{\QQ^k} - \gF{\XX^k}}}^2}  \\
		%\comleq{\eqref{eq:ConYsp1}}& \pr{1 - \gap}^2\pr{1 + \frac{\gap}{2}}\mt{\Con\ZZ^k + \bet\Con\XX^k - \eta\Con\GG^k}^2
		% + \pr{\frac{\eta^2}{q} + \frac{2\eta^2}{\gap}}\nb{\Con}^2\mt{\gF{\QQ^k} - \gF{\XX^k}}^2  \\
		%\comleq{\eqref{eq:Hilbert}}& \pr{1 - \gap}^2\pr{1 + \frac{\gap}{2}}\pr{\frac{1}{1 - \frac{\gap}{2}}\mt{\Con\ZZ^k}^2 + \frac{2\bet^2    }{\gap}\mt{\Con\XX^k}^2 + \frac{2\eta^2}{\gap}\mt{\Con\GG^k}^2}  \\
		%& + \pr{\frac{\eta^2}{q} + \frac{2\eta^2}{\gap}}\mt{\gF{\QQ^k} - \gF{\XX^k}}^2  \\
		\comleq{\eqref{eq:ConYsp1}}& \pr{1 + \bet}^{-2}\pr{1 - \gap}^2\pr{1 + \frac{\gap}{2}}\pr{\frac{1}{1 - \frac{\gap}{2}}\mt{\Con\ZZ^k}^2 + \frac{4\bet^2}{\gap}\mt{\Con\XX^k}^2 + \frac{4\eta^2}{\gap}\mt{\Con\GG^k}^2}  \\
		& + \pr{1 + \bet}^{-2}\pr{1 - \gap}^2\pr{\frac{\eta^2}{q} + \frac{2\eta^2}{\gap}}\mt{\gF{\QQ^k} - \gF{\XX^k}}^2  \\
		\leq& \pr{1 - \gap}\mt{\Con\ZZ^k}^2 + \frac{4\bet^2    }{\gap}\mt{\Con\XX^k}^2 + \frac{4\eta^2}{\gap}\mt{\Con\GG^k}^2  \\
		& + \pr{\frac{\eta^2}{q} + \frac{2\eta^2}{\gap}}\mt{\gF{\QQ^k} - \gF{\XX^k}}^2.
	}
	
	Again, by multiplying $\Con$ on both sides of~\eqref{eq:Y}, and using~\eqref{eq:Hilbert}, we have the following inequality
	\eql{\label{eq:ConY}}{
		\mt{\Con\YY^{k+1}}^2    %= \mt{{\pr{1 - \alp - \tau}\Con\YY^k + \gam\Con\ZZ^{k+1} + \tau\Con\UU^k}}^2 \\
		%\leq& \nb{\Con\WW}^2\mt{\pr{1 - \alp - \tau}\Con\YY^k + \gam\Con\ZZ^{k+1} + \tau\Con\UU^k}^2  \\
		\comleq{\eqref{eq:Hilbert}}& \frac{4 - \tau}{4 - 2\tau}\mt{\Con\XX^k}^2 + \frac{4 - \tau}{\tau}\mt{\gam\pr{\Con\ZZ^{k+1} - \Con\ZZ^k}}^2  \\
		\leq& \frac{4 - \tau}{4 - 2\tau}\mt{\Con\XX^k}^2 + \frac{2\pr{4 - \tau}\gam^2}{\tau}\pr{\mt{\Con\ZZ^{k+1}}^2 + \mt{\Con\ZZ^k}^2}.
	}
	
	From the definition of $\xi^k$, it follows that
	\eql{\label{eq:ConU}}{
		&\Ek{\mt{\Con\UU^{k+1}}^2}
		= \pr{1 - p} \mt{\Con\WW\Con\UU^k}^2 + p \mt{\Con\WW\Con\XX^k}^2  \\
		\leq& \pr{1 - p}\nb{\Con\WW}^2\mt{\Con\UU^k}^2 + p\nb{\Con\WW}^2\mt{\Con\XX^k}^2  \\
		\leq& \pr{1 - \gap}^2\mt{\Con\UU^k}^2 + p\mt{\Con\XX^k}^2.
	}
	Also, by the definition of $\xi^k$,
	\eql{\label{eq:ConQ}}{
		&\Ek{\mt{\Con\QQ^{k+1}}^2} = \pr{1 - p}\mt{\Con\QQ^k}^2 + p\mt{\Con\XX^k}^2.
	}
	
	To bound the consensus error of $\GG^{k+1}$, we have
	\eql{\label{eq:ConG}}{
		&\Ek{\mt{\Con\GG^{k+1}}^2} \\
    =& \pr{1 - p}\mt{\Con\WW\Con\GG^k}^2 + p \mt{\Con\WW\Con\GG^k + \Con\pr{\gF{\XX^k} - \gF{\QQ^k}}}^2  \\
		\leq& \pr{1 - p}\nb{\Con\WW}^2\mt{\Con\GG^k}^2 + 2p\nb{\Con\WW}^2\mt{\Con\GG^k}^2 + 2p\nb{\Con}^2\mt{\gF{\XX^k} - \gF{\QQ^k}}^2  \\
		\leq& \pr{1 - \gap}\mt{\Con\GG^k}^2 + 2p\mt{\gF{\XX^k} - \gF{\QQ^k}}^2,
		%\comleq{\eqref{eq:diffgF}}& \pr{1 - \gap}\mt{\Con\GG^k}^2 + 8pLn\pr{f\pr{\ua^k} - f\pr{\xa^k} - \jr{\da^k, \ua^k - \xa^k}} + 4pL^2\pr{\mt{\Con\XX^k}^2 + \mt{\Con\UU^k}^2},
	}
 {where the second inequality holds because of $\pr{1 + p}\pr{1 - \gap}^2 \leq 1 - \gap$, which can be derived from the first condition in~\eqref{eq:Lyparastp}. }
%	where we used the relation $\pr{1 + p}\pr{1 - \gap}^2 \leq 1 - \gap$ derived from~\eqref{eq:Lyparastp} in the second inequality.
	
	Taking weighted sums on both sides of~\eqref{eq:ConX},~\eqref{eq:ConZ},~\eqref{eq:ConY},~\eqref{eq:ConU},\eqref{eq:ConQ},~\eqref{eq:ConG} yields
	\eq{
		&  \mt{\Con\XX^k}^2 + { \mathbb{E}_k\Big[}\frac{\tau}{4p}{\mt{\Con\QQ^{k+1}}^2
		 + \frac{4\pr{1 - \tau}}{4 - \tau}\mt{\Con\YY^{k+1}}^2}{ \Big] }  \\
        & + \Ek{\frac{16}{\gap}\mt{\Con\UU^{k+1}}^2 + \frac{48\pr{1 - \tau}\gam^2}{\tau\gap}\mt{\Con\ZZ^{k+1}}^2 + \frac{1}{\gap  L^2}\mt{\Con\GG^{k+1}}^2}  \\
		\leq& \pr{\frac{2\pr{1 - \tau}}{2 - \tau} + \frac{192\pr{1 - \tau}\gam^2\bet^2}{\tau\gap^2} + \frac{\tau}{4} + \frac{16p}{\gap}}\mt{\Con\XX^k}^2  \\
		& + \pr{1 - p}\frac{\tau}{4p}\mt{\Con\QQ^k}^2
		+ \frac{4\pr{1 - \tau}^2}{4 - \tau}\mt{\Con\YY^k}^2  \\
		& + \pr{1 - 2\gap + \gap^2 + \frac{\tau\gap}{2}}\frac{16}{\gap}\mt{\Con\UU^k}^2
		+ \pr{1 - \gap + \frac{\gap}{6} + \frac{\alp^2\gap}{6\pr{1 - \tau}\gam^2}}\frac{48\pr{1 - \tau}\gam^2}{\tau\gap}\mt{\Con\ZZ^k}^2  \\
		& + \frac{8\pr{1 - \tau}\gam^2}{\tau}\Ek{\mt{\Con\ZZ^{k+1}}^2}
		+ \pr{1 - \gap       + \frac{192\pr{1 - \tau}\gam^2\eta^2  L^2}{\tau\gap}}\frac{1}{\gap  L^2}\mt{\Con\GG^k}^2  \\
		& + \pa\mt{\gF{\XX^k} - \gF{\QQ^k}}^2  \\
		\comleq{\eqref{eq:Lyparastp}}& \pr{1 - \frac{\tau}{8}}\mt{\Con\XX^k}^2
		+ \pr{1 - p}\frac{\tau}{4p}\mt{\Con\QQ^k}^2
		+ \frac{4\pr{1 - \tau}^2}{4 - \tau}\mt{\Con\YY^k}^2
		+ \pr{1 - \frac{\gap}{2}}\frac{16}{\gap}\mt{\Con\UU^k}^2  \\
		& + \pr{1 - \frac{2\gap}{3}}\frac{48\pr{1 - \tau}\gam^2}{\tau\gap}\mt{\Con\ZZ^k}^2 + \frac{8\pr{1 - \tau}\gam^2}{\tau}\Ek{\mt{\Con\ZZ^{k+1}}^2}  \\
		& + \pr{1 - \frac{\gap}{2}}\frac{1}{\gap  L^2}\mt{\Con\GG^k}^2
		+ \pa\mt{\gF{\XX^k} - \gF{\QQ^k}}^2,
	}
	where
	\eq{
		\pa = \frac{48\pr{1 - \tau}\gam^2}{\tau\gap}\pr{\frac{\eta^2}{q} + \frac{2\eta^2}{\gap}} + \frac{2p}{\gap  L^2}.
	}
	Rearranging the above relation yields
	\eq{
		&  \frac{\tau}{8}\pr{\mt{\Con\XX^k}^2 + \mt{\Con\QQ^k}^2} + \frac{\tau}{4p}\Ek{\mt{\Con\QQ^{k+1}}^2
		 + \frac{4\pr{1 - \tau}}{4 - \tau}\mt{\Con\YY^{k+1}}^2} \\
        & + \Ek{\frac{16}{\gap}\mt{\Con\UU^{k+1}}^2 + \pr{1 - \frac{\gap}{6}}\frac{48\pr{1 - \tau}\gam^2}{\tau\gap}\mt{\Con\ZZ^{k+1}}^2 + \frac{1}{\gap  L^2}\mt{\Con\GG^{k+1}}^2} \\
		\leq&   \pr{1 - \frac{p}{2}}\frac{\tau}{4p}\mt{\Con\QQ^{k}}^2 +  \frac{4\pr{1 - \tau}^2}{4 - \tau}\mt{\Con\YY^{k}}^2  \\
        & + \pr{1 - \frac{\gap}{2}}\frac{16}{\gap}\mt{\Con\UU^{k}}^2 + \pr{1 - \frac{2\gap}{3}}\frac{48\pr{1 - \tau}\gam^2}{\tau\gap}\mt{\Con\ZZ^{k}}^2  \\
		& + \pr{1 - \frac{\gap}{2}}\frac{1}{\gap  L^2}\mt{\Con\GG^{k}}^2 + \pa\mt{\gF{\XX^k} - \gF{\QQ^k}}^2  \\
		\leq& \pr{1 - \min\dr{\frac{p}{2}, \tau, \frac{\gap}{2}}}{\Big(} \frac{\tau}{4p}\mt{\Con\QQ^k}^2    + \frac{4\pr{1 - \tau}}{4 - \tau}\mt{\Con\YY^{k}}^2 + \frac{16}{\gap}\mt{\Con\UU^{k}}^2  \\
		&+ \pr{1 - \frac{\gap}{6}}\frac{48\pr{1 - \tau}\gam^2}{\tau\gap}\mt{\Con\ZZ^{k}}^2
		+ \frac{1}{\gap  L^2}\mt{\Con\GG^{k}}^2{\Big)}  \\
		& + \pa\mt{\gF{\XX^k} - \gF{\QQ^k}}^2,
	}
	%Using the fact that $\pr{1 - \frac{2\gap}{3}} \leq \pr{1 - \frac{\gap}{6}}\pr{1 - \frac{\gap}{2}}$ and
	which leads to~\eqref{eq:LyConX}.
\end{proof}
Lemma~\ref{lem:conbd1} can also be understood intuitively in another way.
Since $\Ly^k$ is the weighted sum of consensus errors, inequality~\eqref{eq:LyConX} also indicates that
the weighted  sum of consensus errors is a ``Q-linear" sequence with ``additional errors" in term of $\pb\mt{\gF{\XX^k} - \gF{\QQ^k}}^2$.
% where $\pb$ is specified in~\eqref{eq:defpb1}.
 By the magnitudes of the parameters and stepsize mentioned in~\eqref{eq:guideparastp1},
 we have \eq{\pb = O\pr{\frac{\gam^2\eta^2}{\gap^2} + \frac{\gam^2\eta^2}{q\gap} + \frac{p}{\gap L^2}} = O\pr{\frac{1}{L^2}}.  }

\subsection{Convergence rate of \SSGT\ }\label{sec:SSGTrate}
In this section, we prove the complexities of \SSGT.
To begin with, we bound the gradient difference $\mt{\gF{\XX^k} - \gF{\QQ^k}}$, which occurs on the RHS of~\eqref{eq:LyConX} and~\eqref{eq:dajrz+1}.
\begin{lemma}\label{lem:diffgF12}
	\eql{\label{eq:diffgF}}{
		\mt{\gF{\XX^k} - \gF{\QQ^k}}^2 \leq& 4Ln\pr{f\pr{\ua^k} - f\pr{\xa^k} - \jr{\da^k, \ua^k - \xa^k}} \\
 & + 2L^2\pr{\mt{\Con\XX^k}^2 + \mt{\Con\QQ^k}^2}.
	}
	
\end{lemma}
\begin{proof}
	First, we decompose
	\eq{
		&\mt{\gF{\XX^k} - \gF{\QQ^k}}^2  \\
		=& \mt{\gF{\XX^k} - \gF{\one\ua^k} + \gF{\one\ua^k} - \gF{\QQ^k}}^2  \\
		\leq& 2\mt{\gF{\XX^k} - \gF{\one\ua^k}}^2 + 2\mt{\gF{\one\ua^k} - \gF{\QQ^k}}^2  \\
		\leq& 2\mt{\gF{\XX^k} - \gF{\one\ua^k}}^2 + 2L^2\mt{\Con\QQ^k}^2,
	}
	where we used the $L$-smoothness of $\dr{f_i}$ and~\eqref{eq:efctua1} in the last inequality.
	
	For any $i\in\calN$,
	\eq{
		&\frac{1}{2L}\nt{\na f_i\pr{\xx^k_i} - \na f_i\pr{\ua^k}}^2 \leq f_i\pr{\ua^k} - f_i\pr{\xx^k_i} - \jr{\na f_i\pr{\xx^k_i}, \ua^k - \xx^k_i}  \\
		=& f_i\pr{\ua^k} - f_i\pr{\xa^k} - \jr{\na f_i\pr{\xx^k_i}, \ua^k - \xa^k}  \\
		& + f_i\pr{\xa^k} - f_i\pr{\xx^k_i} - \jr{\na f_i\pr{\xx^k_i}, \xa^k - \xx^k_i}  \\
		\leq& f_i\pr{\ua^k} - f_i\pr{\xa^k} - \jr{\na f_i\pr{\xx^k_i}, \ua^k - \xa^k} + \frac{L}{2}\mt{\xa^k - \xx^k_i}^2,
	}
	where the first inequality comes from the $L$-smoothness of $f_i$ and Lemma 3.5 of the textbook~\cite{bubeck2015convex}, and
	the last inequality is due to the $L$-smoothness of $f_i$.
	
	Taking average over $i\in \calN$ yields
	\eq{
		&\frac{1}{2Ln}\mt{\gF{\XX^k} - \gF{\one\ua^k}}^2 \leq f\pr{\ua^k} - f\pr{\xa^k} - \jr{\da^k, \ua^k - \xa^k} + \frac{L}{2n}\mt{\Con\XX^k}^2,
	}
	which completes the proof.
\end{proof}

In what follows, the term $f\pr{\ua^k} - f\pr{\xa^k}$ on the RHS of~\eqref{eq:diffgF} is bounded in Lemma~\ref{lem:LyU}.
The inequality~\eqref{eq:dajrxz} in Lemma~\ref{lem:dajrsupp} deals with the term $\jr{\da^k, \xa^k - \ua^k}$ on the RHS of~\eqref{eq:diffgF} due to the existence of the negative momentum $\tau\UU^k$.
The remaining term $\mt{\Con\XX^k}^2 + \mt{\Con\QQ^k}^2$ on the RHS of~\eqref{eq:diffgF} can be bounded by~\eqref{eq:LyConX}.
\begin{lemma}\label{lem:LyU}
	\eql{\label{eq:LyU}}{
		\frac{\tau}{\alp}\pr{f\pr{\ua^k} - f\pr{\xa^k}} \leq \pr{1 - \pra}\LyU^k - \Ek{\LyU^{k+1}} + {\frac{1}{2}}\pr{f\pr{\xa^k} - f\pr{\xx^*}},
	}
	where the Lyapunov function
	\eq{
		\LyU^k = \frac{2\tau + \alp}{2\alp p}\pr{f\pr{\ua^k} - f\pr{\xx^*}}
	}
	and
	\eql{\label{eq:defpra1}}{
		\pra = \frac{\alp p}{2\tau + \alp}.
	}
	
\end{lemma}
\begin{proof}
	Denote $c = \frac{2\tau + \alp}{2\alp p}$ for simplicity.
	By the distribution of $\xi^k$ and~\eqref{eq:ua},
	\eql{\label{eq:LyUrecur1}}{
		\Ek{\LyU^{k+1}} = \pr{1 - p}\LyU^k + c p \pr{f\pr{\xa^k} - f\pr{\xx^*}}.
	}
	Thus,
	\eq{
		&\frac{\tau}{\alp}\pr{f\pr{\ua^k} - f\pr{\xx^*}} = \frac{\tau}{c\alp}\LyU^k  \\
    \comeq{\eqref{eq:LyUrecur1}}& \frac{\tau}{c\alp\pr{p - \pra}}\pr{\pr{1 - \pra}\LyU^k - \Ek{\LyU^{k+1}} + c p \pr{f\pr{\xa^k} - f\pr{\xx^*}}}  \\
		=& \pr{1 - \pra}\LyU^k - \Ek{\LyU^{k+1}} + \pr{\frac{\tau}{\alp} + \frac{1}{2}}\pr{f\pr{\xa^k} - f\pr{\xx^*}}.
	}
	Rearranging the above equation yields~\eqref{eq:LyU}.
\end{proof}

The following supporting lemma is from an adaptation of some standard steps in the proofs of Katyusha~\cite{allen2017katyusha} or Loopless Katyusha~\cite{kovalev2020don} to our methods.
We attach its proof in Appendix~\ref{sec:prflemdajrsupp}.

\begin{lemma}\label{lem:dajrsupp}
	The following inequalities hold.
	\begin{align}
		&\bullet \jr{\da^k, \xa^k - \za^k}
		\leq \frac{1 - \alp - \tau}{\alp}\pr{f\pr{\ya^k} - f\pr{\xa^k}} + \frac{\tau}{\alp}\jr{\da^k, \ua^k - \xa^k} \notag \\
        &\qquad\qquad\qquad\qquad + \frac{\pr{1 - \alp - \tau}L}{\alp n}\mt{\Con\XX^k}^2,
		\label{eq:dajrxz} \\
		&\bullet \Ek{\jr{\ga^k + \zeta^k\pr{\da^k - \ga^k}, \za^k - \za^{k+1}}}
		\leq \frac{1}{\gam}\Ek{f\pr{\xa^k} - f\pr{\ya^{k+1}}} + \frac{L}{n\gam}\mt{\Con\XX^k}^2 \notag \\
        & \qquad\qquad\qquad\qquad + \pr{L\gam + \frac{1}{4\eta}}\Ek{\nt{\za^k - \za^{k+1}}^2} + \frac{\eta}{qn}\mt{\gF{\XX^k} - \gF{\QQ^k}}^2,
		\label{eq:dajrz+1} \\
		&\bullet 2\eta\jr{\ga^k + \zeta^k\pr{\da^k - \ga^k}, \za^{k+1} - \xx^*}
		\leq \nt{\za^k - \xx^*}^2 - \pr{1 + \bet}\nt{\za^{k+1} - \xx^*}^2 \notag  \\
        &\qquad\qquad\qquad\qquad\qquad\qquad\qquad\qquad\quad + \bet\nt{\xa^k - \xx^*}^2 - \nt{\za^k - \za^{k+1}}^2.
		\label{eq:dajrzaxo1}
	\end{align}
\end{lemma}

The next supporting lemma will be used in the proof of Theorem~\ref{thm:SSGT1}. %, which is the main convergence result of \SSGT.
Lemma~\ref{lem:ratesupp2} is derived by linearly coupling~\eqref{eq:LyConX} and~\eqref{eq:diffgF}.

\begin{lemma}\label{lem:ratesupp2}
	Define
	\eq{
		\pc = \pr{\frac{1 - \alp - \tau}{\alp  n} + \frac{1}{n\gam} + \frac{1}{n}}L.
	}
	If the parameters and stepsize satisfy
	\subeqnuml{\label{eq:parastpdiffgF}}{
		\frac{2\eta L^2}{qn} \leq \frac{\pc}{2},\
		\pc\pb \leq \frac{\eta}{2qn},
		\label{eq:paraLyhalf}  \\
		\frac{8\eta L}{q} \leq \frac{\tau}{\alp}, \label{eq:BGtdba1}
	}
	then we have the following inequality:
	\eql{\label{eq:LydiffgF1}}{
		&\pc\mt{\Con\XX^k}^2
		+ \frac{{ \eta}}{qn}\mt{\gF{\XX^k} - \gF{\QQ^k}}^2  \\
		\leq& \frac{\tau}{\alp}\pr{f\pr{\ua^k} - f\pr{\xa^k} - \jr{\da^k, \ua^k - \xa^k}} \\
		& + \pd\pr{\pr{1 - \min\dr{\frac{p}{2}, \tau, \frac{\gap}{2}}}\Ly^k - \Ek{\Ly^{k+1}}},
	}
	where
	\eq{
		\pd = \frac{\pc\tau q}{4\alp \eta  L}\geq 2\pc.
	}
	
\end{lemma}
\begin{proof}
	Using~\eqref{eq:diffgF} and~\eqref{eq:LyConX}, we have
	\eq{
		&\pc\pr{\mt{\Con\XX^k}^2 + \mt{\Con\QQ^k}^2}
		+ \frac{\eta}{qn}\mt{\gF{\XX^k} - \gF{\QQ^k}}^2  \\
		\leq& \pc\pb\mt{\gF{\XX^k} - \gF{\QQ^k}}^2  \\
		& + \pc\pr{\pr{1 - \min\dr{\frac{p}{2}, \tau, \frac{\gap}{2}}}\Ly^k - \Ek{\Ly^{k+1}}}  \\
		& + \frac{2\eta L^2}{qn}\pr{\mt{\Con\XX^k}^2 + \mt{\Con\QQ^k}^2}
		+ \frac{4\eta L}{q}\pr{f\pr{\ua^k} - f\pr{\xa^k} - \jr{\da^k, \ua^k - \xa^k}}  \\
		\comleq{\eqref{eq:paraLyhalf}}& \frac{1}{2}\pr{\pc\pr{\mt{\Con\XX^k}^2 + \mt{\Con\QQ^k}^2}
			+ \frac{\eta}{qn}\mt{\gF{\XX^k} - \gF{\QQ^k}}^2}  \\
		& + \pc\pr{\pr{1 - \min\dr{\frac{p}{2}, \tau, \frac{\gap}{2}}}\Ly^k - \Ek{\Ly^{k+1}}}  \\
		& + \frac{4\eta L}{q}\pr{f\pr{\ua^k} - f\pr{\xa^k} - \jr{\da^k, \ua^k - \xa^k}}.
	}
	Rearranging the above equation yields
	\eq{
		&\pc\pr{\mt{\Con\XX^k}^2 + \mt{\Con\QQ^k}^2}
		+ \frac{\eta}{qn}\mt{\gF{\XX^k} - \gF{\QQ^k}}^2  \\
		\leq& \frac{8\eta L}{q}\pr{f\pr{\ua^k} - f\pr{\xa^k} - \jr{\da^k, \ua^k - \xa^k}}  \\
		& + 2\pc\pr{\pr{1 - \min\dr{\frac{p}{2}, \tau, \frac{\gap}{2}}}\Ly^k - \Ek{\Ly^{k+1}}}  \\
		\comleq{\eqref{eq:BGtdba1}}& \frac{\tau}{\alp}\pr{f\pr{\ua^k} - f\pr{\xa^k} - \jr{\da^k, \ua^k - \xa^k}} \\
		& + \pd\pr{\pr{1 - \min\dr{\frac{p}{2}, \tau, \frac{\gap}{2}}}\Ly^k - \Ek{\Ly^{k+1}}}.
	}
	Then,~\eqref{eq:LydiffgF1} follows from $\mt{\Con\QQ^k}^2 \geq 0$.
\end{proof}
\begin{remark}
    To see why the requirements~\eqref{eq:paraLyhalf} and~\eqref{eq:BGtdba1} can be satisfied, we can substitute the magnitudes of the parameters and the stepsize in~\eqref{eq:guideparastp1} and find that the magnitudes on both sides of the inequalities in~\eqref{eq:paraLyhalf} and~\eqref{eq:BGtdba1} are actually all the same.
    But the power numbers of $\eta$ or $p$ on the LHS is higher than those on  the RHS.
    Hence we can choose the other parameters first and  select $p, \eta$ to satisfy the inequalities at last.
    In this way, we can find the proper parameters and stepsizes for inequalities~\eqref{eq:paraLyhalf} and~\eqref{eq:BGtdba1}.
\end{remark}

Now, we formally state the complexities of $\SSGT$.
\begin{theorem}\label{thm:SSGT1}
	Define the Lyapunov function
	\eq{
		\Lya^k = \frac{1}{\gam}\pr{f\pr{\ya^{k}} - f\pr{\xx^*}} + \LyU^{k} + \frac{1+\bet}{2\eta}\nt{\za^{k} - \xx^*}^2 + \pd\Ly^{k}
	}
    and a constant
    \eq{
        \rate = \min\dr{\frac{\gam}{4}, \pra, \frac{\bet}{1+\bet}, \frac{p}{2}, \tau, \frac{\gap}{2}},
    }
    where $\pra$ is given by~\eqref{eq:defpra1}.

	If the parameters and stepsize satisfy~\eqref{eq:Lyparastp},~\eqref{eq:parastpdiffgF} and
	\subeqnuml{\label{eq:thmsnapshotGTcond1}}{
		L\gam \leq \frac{1}{4\eta},\ \frac{\bet}{2\eta} \leq \frac{\mu}{4},  \label{eq:elimstrcv1}  \\
		\frac{1}{\gam} - \frac{1 - \alp - \tau}{\alp} - \frac{1}{2} \leq 0, \label{eq:elimfxk1}  \\
		\frac{1 - \alp - \tau}{\alp} \leq \frac{1}{\gam} - \frac{1}{4}, \label{eq:upbfyk1}
	}
	then $\E{\Lya^k}$ converges linearly with
	\eq{
		\Ek{\Lya^{k+1}} \leq \pr{1 - \rate}\Lya^k.
	}
	Specifically, we can choose
	\eql{\label{eq:parastpex1}}{	
		\tau = \frac{1}{2},\
		\alp = \frac{1}{23\sqrt{\ka}},\
		\gam = \frac{4\alp}{4 - 4\tau - 3\alp},\
		p = q = \frac{\gap}{4232},\
		\eta = \frac{\gap\sqrt{\ka}}{5000  L},\
		\bet = \frac{\mu\eta}{2}.
	}
	In this case, to achieve $\E{\Lya^K} \leq \eps\Lya^0$,
	the gradient computation complexity is $O\pr{\sqrt{\ka}\log\frac{1}{\eps}}$
	and the communication complexity is $O\pr{\frac{\sqrt{\ka}}{\gap}\log\frac{1}{\eps}}$.

\end{theorem}
\begin{proof}
	We have
	\eq{
		&f\pr{\xa^k} - f\pr{\xx^*} \\
		\comleq{\eqref{eq:inextmu}}& \jr{\da^k, \xa^k - \xx^*} - \frac{\mu}{4}\nt{\xa^k - \xx^*}^2 + \frac{L}{n}\mt{\Con\XX^k}^2  \\
		=& \jr{\da^k, \xa^k - \za^k} + \Ek{\jr{\ga^k + \zeta^k\pr{\da^k - \ga^k}, \za^k - \xx^*}} - \frac{\mu}{4}\nt{\xa^k - \xx^*}^2 + \frac{L}{n}\mt{\Con\XX^k}^2  \\
		=& \jr{\da^k, \xa^k - \za^k} + \Ek{\jr{\ga^k + \zeta^k\pr{\da^k - \ga^k}, \za^k - \za^{k+1}}} \\
        & + \Ek{\jr{\ga^k + \zeta^k\pr{\da^k - \ga^k}, \za^{k+1} - \xx^*}  }
		  - \frac{\mu}{4}\nt{\xa^k - \xx^*}^2 + \frac{L}{n}\mt{\Con\XX^k}^2.
	}
	Substituting the inequalities in Lemma~\ref{lem:dajrsupp} into the above relation and taking conditional expectation yields
	\eql{\label{eq:subst1}}{
		&f\pr{\xa^k} - f\pr{\xx^*} \\
		\leq& \frac{1 - \alp - \tau}{\alp}\pr{f\pr{\ya^k} - f\pr{\xa^k}} + \frac{\tau}{\alp}\jr{\da^k, \ua^k - \xa^k}  + \frac{1}{\gam}\Ek{f\pr{\xa^k} - f\pr{\ya^{k+1}}} \\
		& + \pr{\frac{1 - \alp - \tau}{\alp  n} + \frac{1}{n\gam} + \frac{1}{n}}L\mt{\Con\XX^k}^2
		+ \frac{\eta}{qn}\mt{\gF{\XX^k} - \gF{\QQ^k}}^2  \\
		& + \frac{1}{2\eta}\nt{\za^k - \xx^*}^2 - \frac{1 + \bet}{2\eta}\Ek{\nt{\za^{k+1} - \xx^*}^2} + \pr{L\gam + \frac{1}{4\eta} - \frac{1}{2\eta}}\Ek{\nt{\za^k - \za^{k+1}}^2}   \\
        & + \pr{\frac{\bet}{2\eta} - \frac{\mu}{4}}\nt{\xa^k - \xx^*}^2  \\
		\comleq{\eqref{eq:elimstrcv1}}& \frac{1 - \alp - \tau}{\alp}\pr{f\pr{\ya^k} - f\pr{\xa^k}} + \frac{\tau}{\alp}\jr{\da^k, \ua^k - \xa^k}  + \frac{1}{\gam}\Ek{f\pr{\xa^k} - f\pr{\ya^{k+1}}} \\
		& + c_3(\alpha,\gamma,\tau)\mt{\Con\XX^k}^2
		+ \frac{\eta}{qn}\mt{\gF{\XX^k} - \gF{\QQ^k}}^2  \\
		& + \frac{1}{2\eta}\nt{\za^k - \xx^*}^2 - \frac{1 + \bet}{2\eta}\Ek{\nt{\za^{k+1} - \xx^*}^2}  \\
		\comleq{\eqref{eq:LydiffgF1}}& \frac{1 - \alp - \tau}{\alp}\pr{f\pr{\ya^k} - f\pr{\xa^k}}  + \frac{1}{\gam}\Ek{f\pr{\xa^k} - f\pr{\ya^{k+1}}} + \frac{\tau}{\alp}\pr{f\pr{\ua^k} - f\pr{\xa^k}} \\
		& + \frac{1}{2\eta}\nt{\za^k - \xx^*}^2 - \frac{1 + \bet}{2\eta}\Ek{\nt{\za^{k+1} - \xx^*}^2}  \\
        & + \pd\pr{\pr{1 - \min\dr{\frac{p}{2}, \tau, \frac{\gap}{2}}}\Ly^k - \Ek{\Ly^{k+1}}}  \\
		\comleq{\eqref{eq:LyU}}& \frac{1 - \alp - \tau}{\alp}\pr{f\pr{\ya^k} - f\pr{\xa^k}}  + \frac{1}{\gam}\Ek{f\pr{\xa^k} - f\pr{\ya^{k+1}}} + \frac{1}{2}\pr{f\pr{\xa^k} - f\pr{\xx^*}}  \\
		& + \pr{1 - \pra}\LyU^k - \Ek{\LyU^{k+1}}
		+ \frac{1}{2\eta}\nt{\za^k - \xx^*}^2 - \frac{1 + \bet}{2\eta}\Ek{\nt{\za^{k+1} - \xx^*}^2}  \\
		& + \pd\pr{\pr{1 - \min\dr{\frac{p}{2}, \tau, \frac{\gap}{2}}}\Ly^k - \Ek{\Ly^{k+1}}}.
	}
	Rearranging the above equation yields
	\eq{
		&\Ek{\Lya^{k+1}}  \\
        =& \Ek{\frac{1}{\gam}\pr{f\pr{\ya^{k+1}} - f\pr{\xx^*}} + \LyU^{k+1} + \frac{1+\bet}{2\eta}\nt{\za^{k+1} - \xx^*}^2 + \pd\Ly^{k+1}}  \\
		\leq&\frac{1 - \alp - \tau}{\alp}\pr{f\pr{\ya^k} - f\pr{\xx^*}} + \pr{1 - \pra}\LyU^k
		+ \frac{1}{2\eta}\nt{\za^k - \xx^*}^2  \\
        & + \pd\pr{1 - \min\dr{\frac{p}{2}, \tau, \frac{\gap}{2}}}\Ly^k
		  + \pr{\frac{1}{\gam} - \frac{1 - \alp - \tau}{\alp} - \frac{1}{2}}\pr{f\pr{\xa^k} - f\pr{\xx^*}}  \\
		\comleq{\eqref{eq:elimfxk1},\eqref{eq:upbfyk1}}& \pr{\frac{1}{\gam} - \frac{1}{4}}\pr{f\pr{\ya^k} - f\pr{\xx^*}} + \pr{1 - \pra}\LyU^k
		+ \frac{1}{2\eta}\nt{\za^k - \xx^*}^2  \\
        & + \pd\pr{1 - \min\dr{\frac{p}{2}, \tau, \frac{\gap}{2}}}\Ly^k  \\
		\leq& \pr{1 - \rate}\pr{\frac{1}{\gam}\pr{f\pr{\ya^{k}} - f\pr{\xx^*}} + \LyU^{k} + \frac{1+\bet}{2\eta}\nt{\za^{k} - \xx^*}^2 + \pd\Ly^{k}} \\
		=& \pr{1 - \rate}\Lya^k.
	}
	
	Specifically, when the parameters and the stepsize are chosen as in~\eqref{eq:parastpex1},  we have $\rate = O\pr{\frac{\gap}{\sqrt{\ka}}}$, and
	 the desired communication complexity follows.
	Notice that in each communication round, the gradients are computed for at most ${p + q} = O\pr{\gap}$ times in expectation.
	We obtain the corresponding gradient computation complexity.
\end{proof}

%\section{Optimal Gradient Tracking Over Undirected Graphs}\label{sec:UOGT}
\section{Optimal Gradient Tracking  }\label{sec:OGT1}
In this section, we develop \emph{Optimal Gradient Tracking} (\OGT), which improves the dependency on the graph condition number in the communication complexity of $\SSGT$ from $O\pr{\frac{1}{\gap}}$ to $O\pr{\frac{1}{\sqrt{\gap}}}$, while  keeping the gradient computation complexity unchanged.
Thus, $\OGT$ achieves both optimal gradient computation complexity and optimal communication complexity.
Such an improvement relies on an important technique we develop in this section called the \lCA\ (LCA).

In the rest of this section, we first introduce the OGT algorithm. Then, we motivate the development of LCA and provide the analysis for $\OGT$.

\subsection{Algorithm}
%Before introducing $\OGT$,
To begin with, we introduce some useful notations.
% we first define some useful notations.
Define
\eql{\label{eq:defjing1}}{
	\ert = \frac{1 - \sqrt{1 - \pr{1 - \gap}^2}}{1 + \sqrt{1 - \pr{1 - \gap}^2}},\
	\ew = \frac{1 + \ert}{2},\
	\rw = \sqrt{\ew},\
	\cgap = 1 - \rw.
}
Then we can easily see that $\cgap=O(\sqrt{\gap})$.
%And
We use \emph{``\ca"} as a subscript to denote the following special type of vectors and matrices, which are called \catype\ vectors and matrices.
For any $n$-dimensional vector $\vv$, we define $\vv_{\ca}$ as a $2n$-dimensional vector as follows:
\eq{
	\vv_{\ca} =
	\begin{pmatrix}
		\vv \\
		\vv
	\end{pmatrix}.
}
For any $n$-by-$d$ matrix $\AA$, we define $\AA_{\ca}$ as a $2n$-by-$d$ matrix as follows:
\eq{
	\AA_{\ca} =
	\begin{pmatrix}
		\AA  \\
		\AA
	\end{pmatrix}.
}
Regarding the gradients, we denote
\eq{
	\cgF{\XX^k} =
	\begin{pmatrix}
		\gF{\XX^k} \\
		\gF{\XX^k}
	\end{pmatrix},\
	\cgF{\QQ^k} =
	\begin{pmatrix}
		\gF{\QQ^k} \\
		\gF{\QQ^k}
	\end{pmatrix}.
}
It follows from the definitions that $\mt{\cgF{\XX^k}}^2 = 2\mt{\gF{\XX^k}}^2$,
%$\mt{\cgF{\QQ^k}}^2 = 2\mt{\gF{\QQ^k}}^2$,
$\mt{\AA_{\ca}}^2 = 2\mt{\AA}^2$, and $\nt{\vv_{\ca}}^2 = 2\nt{\vv}^2$.

Define a $2n$-by-$2n$ augmented matrix for the gossip matrix $\WW$ and $\acon$ as
\eql{\label{eq:augW1}}{
	\Wt =
	\begin{pmatrix}
		\pr{1 + \ew}\WW & -\ew\II \\
		\II & \zero
	\end{pmatrix},\qquad \acon =
	\begin{pmatrix}
		\Con & \zero \\
		\zero & \Con
	\end{pmatrix}.
}
It is obvious that
\eq{
	\acon\AA_{\ca} = \pr{\Con\AA}_{\ca},\ \acon\Wt = \acon\Wt\acon,\ \nb{\acon} = 1.
}

The $2n$-by-$d$ matrices used in $\OGT$ are concatenations of $2n$ vectors in the following way (here, we take $\Zt^k$ as an example):
\eq{
	\Zt^k =
	\begin{pmatrix}
		\frac{ \quad }{ } & \ztil^k_{1, 1} & \frac{ \quad }{ }  \\
		\frac{ \quad }{ } & \ztil^k_{2, 1} & \frac{ \quad }{ }  \\
		& \vdots &  \\
		\frac{ \quad }{ } & \ztil^k_{n, 1} & \frac{ \quad     }{ }  \\
		\frac{ \quad }{ } & \ztil^k_{1, 2} & \frac{ \quad }{ }  \\
		\frac{ \quad }{ } & \ztil^k_{2, 2} & \frac{ \quad }{ }  \\
		& \vdots &  \\
		\frac{ \quad }{ } & \ztil^k_{n, 2} & \frac{ \quad     }{ }
	\end{pmatrix},
}
where $\ztil^k_{i,1}, \ztil^k_{i,2} \in \Real^d$ belong to agent $i$.

With the above notations, we can present the $\OGT$ method now.
The $\OGT$ method starts with the initial values:
\eql{\label{eq:initQ}}{
    \YY^0 = \ZZ^0 = \UU^0 = \QQ^0 = \XX^0,\ \Zt^0 = \Ut^0 = \CX^0,\ \Gt^0 = \cgF{\XX^0  }
}
and updates as follows:
\subeqnuml{\label{eq:UOGT1}}{
	%\YY^0 = \ZZ^0 = \UU^0 = \QQ^0 = \XX^0,\ \Zt^0 = \Ut^0 = \CX^0,\ \Gt^0 = \cgF{\XX^0  } \label{eq:initQ}  \\
	\XX^k = \pr{1 - \alp - \tau}\YY^k + \alp\ZZ^k + \tau\UU^k \label{eq:CX} \\
	\Zt^{k+1} = \pr{1 + \bet}\inv\Wt\Big(\Zt^k + \bet\CX^k  \\
    \qquad  - \eta\CG^k + \eta\zeta^k\pr{\cgF{\QQ^k} - \cgF{\XX^k}}\Big)  \label{eq:CZ} \\
	\YY^{k+1} = \XX^k + \gam\pr{\ZZ^{k+1} - \ZZ^k} \label{eq:CY} \\
	\QQ^{k+1} = \pr{1 - \xi^k}\QQ^k + \xi^k\XX^k \label{eq:CQ}  \\
	\Ut^{k+1} = \Wt\pr{\pr{1 - \xi^k}\Ut^k + \xi^k\CX^k}  \label{eq:CU}  \\
	\Gt^{k+1} = \Wt\Gt^k + \xi^k\pr{\cgF{\XX^k} - \cgF{\QQ^k}}  \label{eq:CG}  % \\
}
where
 $\dr{\xi^k, \zeta^k}$ are independent random variables with $\xi^k \sim Bernoulli\pr{p}   $, $\zeta^k \sim Bernoulli\pr{q}/q  $
and
\eq{
	\ZZ^k = \ub{\Zt^k},\ \UU^k = \ub{\Ut^k},\ \GG^k = \ub{\Gt^k}.
}
Recalling the definition of \catype\ matrices, we write
\eq{
	\CX^k =
	\begin{pmatrix}
		\XX^k \\
		\XX^k
	\end{pmatrix},\
	\CG^k =
	\begin{pmatrix}
		\GG^k \\
		\GG^k
	\end{pmatrix}
	=
	\begin{pmatrix}
		\ub{\Gt^k} \\
		\ub{\Gt^k}
	\end{pmatrix}.
}

An implementation-friendly version of $\OGT$ is illustrated in Algorithm~\ref{alg:UOGT} (see Appendix~\ref{sec:ifDOGTandUOGT}).
The equivalence between~\eqref{eq:UOGT1}
%~\eqref{eq:CX}-\eqref{eq:CG}
and Algorithm~\ref{alg:UOGT} can be seen easily from the observation that $\cgF{\QQ^k}$ in~\eqref{eq:UOGT1}
%~\eqref{eq:CX}-\eqref{eq:CG}
equals $\MM^k_{\ca}$ in Algorithm~\ref{alg:UOGT}.
\begin{remark}
	Except $\Wt$ and $\acon$, any bold upper-case letter with ``$\sim$" overhead  denotes a $2n$-by-$d$ matrix.
	The same bold upper-case letter with ``$\sim$" removed denotes an $n$-by-$d$ matrix containing the first $n$ rows of the $2n$-by-$d$ matrix.
	For instance, $\ZZ^k \in \MatSize{n}{d}$, $\Zt^k\in \MatSize{2n}{d}$ and $\ZZ^k = \ub{\Zt^k}.  $
	
\end{remark}

\newcommand\at{\widetilde{\aa}}
\newcommand\At{\widetilde{\AA}}
\begin{remark}\label{rem:Wtmultiat1}
	Multiplication with $\Wt$ can be computed by each agent in the following way (taking $\At = \Wt\Zt^k$ as an example):
	each agent $i$ sends $\ztil^k_{i,1}$ to its neighbors.
	Then, each agent $i$ computes \eq{\at_{i,1} = \pr{1 + \ew}\sum_{j\in\calN_i\cup\{i\}}\WW_{ij}\ztil^k_{j,1} - \ew\ztil^k_{i,2},    }
	where $\calN_i$ is the neighbor set of agent $i$ in $\WW$.
	Each agent $i$ then sets \eq{\at_{i,2} = \ztil^k_{i,1}. }
\end{remark}
\begin{remark}\label{rem:senvecnumOptGT1}
    To implement matrix multiplication with $\Wt$ as in Remark~\ref{rem:Wtmultiat1}, each agent needs to send only one  vector of length $d$ to its neighbors in $\WW$.
    Therefore, in one communication round of $\OGT$, each agent needs to send $3$ vectors of length $d$ to its neighbors.

\end{remark}
\begin{remark}
	The $\CG^k$ added into $\Zt^{k+1}$ in~\eqref{eq:CZ} can be replaced by $\Gt^k$ with the same complexities guaranteed, i.e., replacing~\eqref{eq:CZ}
	by
	\begin{equation}\label{eq:ZtaddebyGt1}%\leqno{\ref{eq:CZ}'}
		\Zt^{k+1} = \pr{1 + \bet}\inv\Wt\pr{\Zt^k + \bet\CX^k - \eta\Gt^k + \eta\zeta^k\pr{\cgF{\QQ^k} - \cgF{\XX^k}}}.
	\end{equation}
	The proof when we use~\eqref{eq:ZtaddebyGt1} instead of~\eqref{eq:CZ} is quite similar with only little changes to Lemma~\ref{lem:qtorc1}.
	In addition, there is no significant difference of numerical performance between these two cases.
	Thus we omit the details of the case where~\eqref{eq:CZ} is replaced by~\eqref{eq:ZtaddebyGt1}.
\end{remark}

In the proof of $\OGT$, we still denote
\eq{
	\za^k = \frac{1}{n}\one\tp\ZZ^k,\ \ua^k = \frac{1}{n}\one\tp\UU^k,\ \ga^k = \frac{1}{n}\one\tp\GG^k.
}
Recall that $\ZZ^k, \UU^k, \GG^k$
are the first $n$-rows of $\Zt^k, \Ut^k, \Gt^k$.
And the notations $\xa^k, \ya^k, \qa^k$ have the same meaning as in Section~\ref{sec:SSGT1}.
We also define
\eq{
	\zat^k = \frac{1}{2n}\one\tp\Zt^k,\ \uat^k = \frac{1}{2n}\one\tp\Ut^k,\ \gat^k = \frac{1}{2n}\one\tp\Gt^k.
}
The next lemma characterizes the evolution of the average parts.
\begin{lemma}\label{lem:caevol1}
	\subeqnuml{\label{eq:caupdt1}}{
		\xa^{k} = \pr{1 - \alp - \tau}\ya^k + \alp\za^k + \tau\ua^k \label{eq:cxa}  \\
		\za^{k+1} = \pr{1 + \bet}\inv\pr{\za^k + \bet\xa^k - \eta\ga^k + \zeta^k\eta\pr{\ga^k - \da^k}} \label{eq:cza}  \\
		\ya^{k+1} = \xa^k + \gam\pr{\za^{k+1} - \za^k} \label{eq:cya} \\
		\qa^{k+1} = \ua^{k+1} = \xi^{k}\xa^k + \pr{1 - \xi^k}\ua^k  \label{eq:cqa}  \\
		\ga^{k+1} = \frac{1}{n}\one\tp\gF{\QQ^{k+1}} \label{eq:cga}
	}
	and
	\eql{\label{tiloleq1}}{
		\zat^k = \za^k,\ \uat^k = \ua^k,\ \gat^k = \ga^k.
	}
	
\end{lemma}
\begin{proof}
	See Appendix~\ref{sec:prflemcaevol1}
\end{proof}

By~\eqref{tiloleq1}, we have the following decomposition (taking $\Zt^k$ as an example):
\eq{
	\Zt^k = \acon\Zt^k + \one\za^k.
}
Here, we call $\mt{\acon\Zt^k}^2$ the consensus error of $\Zt^k$.
Since~\eqref{eq:cxa}-\eqref{eq:cga} are exactly the same as~\eqref{eq:xa}-\eqref{eq:ga}, the analysis for the average parts is similar to that in Section~\ref{sec:SSGT1}.
Therefore, the most important part of analysis is to bound the consensus errors.

\subsection{Loopless Chebyshev Acceleration (LCA)}
Compared to \SSGT, the \OGT\ method improves the dependency of the graph condition number in communication complexity from $O\pr{\frac{1}{\gap}}$ to $O\pr{\frac{1}{\sqrt{\gap}}}$.
Previous works rely on  inner loops of Chebyshev acceleration to reduce the graph condition number in the complexities, see for instance~\cite{kovalev2020optimal,li2021accelerated,scaman2017optimal}.
However, due to the reasons stated in Section~\ref{sec:needlcal1}, it is hard to naively implement Chebyshev acceleration without inner loops in decentralized algorithms.
The augmented matrix defined in~\eqref{eq:augW1} is used by~\cite{liu2011accelerated} to achieve faster consensus.
However, it may happen that $\nb{\acon\Wt} \geq 1$!
This prevents us from analyzing the effect of $\Wt$ in a similar way of analyzing the gossip matrix $\WW$.
The challenges and our strategy to overcome them are stated in the section below in detail.

\subsubsection{Technical challenges in loopless Chebyshev
acceleration}\label{sec:needlcal1}
Chebyshev acceleration (CA, see for instance~\cite{saad1984chebyshev}) over networks was firstly used in~\cite{scaman2017optimal} to achieve optimal complexities for decentralized optimization with dual information.
{After running CA inner loops each for $\lfloor{1\over\sqrt{\theta}}\rfloor$ iterations, the original dependence of the gradient computation complexity on $\theta$ is removed since this operation is equivalent to applying a new gossip matrix with a constant condition number less than $4$ in the algorithm. This technique has been applied to obtain the optimal complexities for several algorithms~\cite{scaman2017optimal,kovalev2020optimal,li2021accelerated}.}
However, a division operation is required at the end of each inner loop. Without such a division,
CA schemes have to use parameters varying with iterations, see for instance~\cite{young2014iterative}.
Both the division operation and the iteration-varying parameters prevent these methods from being implemented ``looplessly".

Another way to improve the communication efficiency of decentralized optimization algorithms is to use an augmented matrix $\Wt$.
It was proved in~\cite{liu2011accelerated} that $\Wt$ has a spectral radius of $1 - O\pr{\sqrt{\gap}}$.
Note that %for the common $2$-norm,
we may have $\nb{\acon\Wt} \geq 1$, which is undesirable.
A typical way to utilize the spectral radius is to find a specific vector norm $\nm{\cdot}_*$ that induces a matrix norm $\nm{\cdot}_*$ which satisfies $\nm{\Wt}_* \simeq 1 - O\pr{\sqrt{\gap}}$.
Such kind of norms do exist; see for instance Lemma 4 in~\cite{pu2020push}.
However, if we define $\alp_1 = \inf_{\vv\neq \zero} \frac{\nm{\vv}_*}{\nt{\vv}}$ and $\alp_2 = \sup_{\vv\neq \zero} \frac{\nm{\vv}_*}{\nt{\vv}}$, the condition number $\frac{\alp_2}{\alp_1}$ will occur in the complexities of the algorithms.
Without additional requirements on the gossip matrix, it may be impossible to  guarantee that the norm $\nm{\cdot}_*$ satisfies $\nm{\Wt}_* \simeq 1 - O\pr{\sqrt{\gap}}$ and $\frac{\alp_2}{\alp_1}$ is independent of $\frac{1}{\gap}$  simultaneously.

Apparently, new approaches are needed to make CA ``loopless".
Observe that, although it is hard to analyze $\Wt^k\uu$ for general $\uu\in \Real^{2n}$, the term $\Wt^k\vv_{\ca} \ (\vv\in \Real^n)$ has good analytic properties.
In fact, $\Wt^k\vv_{\ca} $ can be represented as
$
\begin{pmatrix}
	P_{k}\pr{\WW}\vv \\
	P_{k-1}\pr{\WW}\vv
\end{pmatrix},
$
where $P_{k}\pr{\cdot}$ and $P_{k-1}\pr{\cdot}$ are polynomials of degree $k$ and $k-1$.
When $\WW$ is symmetric, the 2-norm $\nb{P_{k}\pr{\WW}}$ equals $\max\limits_{1\leq i\leq n} \abs{ P_{k}(\lambda_i) } $, where $\la_i \ (1\leq i\leq n)$ are eigenvalues of $\WW$.
This means that we can bound  $\nt{\Wt^k\vv_{\ca}}$ by analyzing $P_{k}\pr{\cdot}$ and $P_{k-1}\pr{\cdot}$ carefully.
Even though  $\dr{\nt{\Wt^k\uu}}_{k \geq 0}$ may not be a Q-linear sequence for general $\uu\in \Real^{2n}$,
our \lCAL\ will show that for any $\catype$ vector $\vv_{\ca}\in \Real^{2n}$, $\dr{\nt{\Wt^k\vv_{\ca}}}_{k \geq 0}$ is always an R-linear sequence.
Notice that by expanding~\eqref{eq:cxa}-\eqref{eq:cga} recursively,
$\XX^k$ and the other variables at iteration $k$ can be represented by linear combinations of the matrices in the form of $\Wt^j\AA_{\ca} (0 \leq j \leq k)$.
By applying the \lCAL\ (Lemma~\ref{lem:lcal1}) developed below, we are able to show in Lemma~\ref{lem:qtorc1} that $\E{\mt{\Con\XX^k}^2}$ and the other consensus errors are R-linear sequences with ``additional errors".
To deal with these ``additional errors", we can use similar methods as in Section~\ref{sec:SSGT1} and then prove the complexities of $\OGT$.

\subsubsection{``Loopless Chebyshev acceleration lemma"  }

The next lemma is the most important lemma for proving the complexities of $\OGT$.
It shows that the {Frobenius norm} of $\acon\Wt^k\acon\AA_{\ca}$ is an R-linear sequence.
More importantly, the constant $\pe = 14$ below is independent of $\frac{1}{\gap}, n$ and any other quantities.
This enables us to overcome the challenges mentioned in Section~\ref{sec:needlcal1}.
\begin{lemma}\label{lem:lcal1}
	(\LCAL)
	%If $\etat \in (\ert, 1)$,
	If $\etat \in [\frac{1 + \ert}{2}, 1)$,
	define a sequence of polynomials
	\eq{
		&T_0\pr{x} = 1,\ T_1\pr{x} = 1,\\
		&T_{k+2}\pr{x} = \pr{1 + \etat}x T_{k+1}\pr{x} - \etat T_k\pr{x},\ k \geq 0.
	}
	Then, we have
	\eql{\label{eq:lcal1}}{
		\sup_{k\geq 0} \sup_{x\in [0, 1 - \gap]} \frac{{T_k\pr{x}}^2}{\etat^k} \leq \pe'.
	}
	Under Assumption~\ref{assp:W},
	for any $\AA\in \MatSize{n}{d}$ and $k \geq 0$,
	\eql{\label{eq:Chebyssp1}}{
		\mt{\acon\Wt^k\acon\AA_{\ca}}^2 \leq \pe\rw^{2k  }\mt{\Con\AA}^2,
	}
	where $\pe' = 7$ and $\pe = 14$.
	
\end{lemma}
\begin{proof}
	Define the function \eq{ r(y) = \pr{1 - \gap}^2\pr{1 + y}^2 - 4y.  }
	By the definition of $\ert$ in~\eqref{eq:defjing1}, we have
	%$\ert$ is a root of $r(y)$.
	\eql{\label{eq:rert01}}{
		r\pr{\ert} = 0.
	}
	Since $r(-1)=4 > 0$, $r(1)=4((1-\theta)^2-1) < 0$ and $r\pr{y}$ is a quadratic polynomial, we have $r(y) < 0$, $\forall \ert < y < 1 .  $
	Since $\etat \in [\frac{\ert + 1}{2}, 1) \subset (\ert, 1)$, we have
	\eql{\label{eq:discriminantneg1}}{
		x^2\pr{1 + \etat}^2 - 4\etat \leq r\pr{\etat} < 0,\ \forall x\in [0, 1 - \gap].
	}
	
	Now, fix an $x \in [0, 1 - \gap]$,
	denote $b_k = T_k\pr{x}$ for simplicity.
	Then, the sequence $\dr{b_k}_{k \geq 0}$ is a linear recurrence with
	\eq{
		&b_0 = b_1 = 1,\\
		&b_{k+2} = x \pr{1 + \etat} b_{k+1} - \etat b_k,\ k \geq 0.
	}
	The characteristic equation of this {linear recurrence} is
	\eq{
		g\pr{a} = a^2 - x\pr{1 + \etat}a + \etat.
	}
	The discriminant of $g\pr{a}$ is
	\eq{
		x^2\pr{1 + \etat}^2 - 4\etat \overset{\eqref{eq:discriminantneg1}}{<} 0.
	}
	Therefore, the {linear recurrence} $\dr{b_k}_{k\geq 0}$ can be solved as
	\eq{
		b_{k} = \frac{1 - \rb}{\ra - \rb}\pr{\ra}^{k} + \frac{1 - \ra}{\rb - \ra}\pr{\rb}^{k},\ \forall k \geq 0,
	}
	where
	\eql{\label{eq:rootceq}}{
		&\ra = \sqrt{\etat - z^2} + i z,\
		\rb = \sqrt{\etat - z^2} - i z.
	}
	Here,
	\eq{
		z = \frac{1}{2}\sqrt{4\etat - x^2\pr{1+\etat}^2}.
	}
	Next, we bound $\abs{\frac{1 - \ra}{\ra - \rb}}^2$.
	Firstly, we give a lower bound for $z$.
	Since $r(y)$ is convex, we have
	\eq{
		r\pr{\etat} =& r\pr{\frac{\etat - \ert}{1 - \ert}\cdot 1 + \frac{1 - \etat}{1 - \ert}\cdot \ert} \leq \frac{\etat - \ert}{1 - \ert} r\pr{1} + \frac{1 - \etat}{1 - \ert} r\pr{\ert}  \\
		\comeq{\eqref{eq:rert01}}& -\frac{\pr{\etat - \ert}  \pr{8\gap - 4\gap^2}}{1 - \ert}
		\leq   - 4\gap + 2\gap^2,
	}
	where
	the last inequality is from $\etat\in [\frac{\ert + 1}{2}, 1).  $
	
	Thus,
	\eql{\label{eq:zlowb1}}{
		z = \frac{1}{2}\sqrt{4\etat - x^2\pr{1+\etat}^2} \geq \frac{1}{2}\sqrt{-r\pr{\etat}} \geq \sqrt{\gap - \frac{\gap^2}{2}}.
	}
	Then, we give lower bounds for $\ert$ and $\etat$:
	\eql{\label{eq:ertlowb1}}{
		\ert &= \frac{1 - \sqrt{1 - \pr{1 - \gap}^2}}{1 + \sqrt{1 - \pr{1 - \gap}^2}} \geq \pr{1 - \sqrt{1 - \pr{1 - \gap}^2}}^2 \geq 1 - 2\sqrt{2\gap - \gap^2}. \\
		\etat &\geq \frac{1 + \ert}{2} \geq 1 - \sqrt{2\gap - \gap^2}.
	}
	
	By~\eqref{eq:rootceq}, the term $\abs{\frac{1 - \ra}{\ra - \rb}}^2$ can be rewritten as
	\eq{
		\abs{\frac{1 - \ra}{\ra - \rb}}^2 = \frac{\abs{1 - \ra}^2}{\abs{\ra - \rb}^2} = \frac{\pr{1 - \sqrt{\etat - z^2}}^2 + z^2}{4z^2}.
	}
	Notice that $1 - \sqrt{\etat - z^2} \leq 1 - \etat + z^2$, we can derive
	\eq{
		&\abs{\frac{1 - \ra}{\ra - \rb}}^2 \leq \frac{\pr{1 - \etat + z^2}^2 + z^2}{4z^2} \leq \frac{2\pr{1 - \etat}^2 + 2z^4 + z^2}{4z^2}  \\
		\comleq{\eqref{eq:zlowb1}}& \frac{\pr{1 - \etat}^2}{2\gap - \gap^2} + \frac{z^2}{2} + \frac{1}{4}
		\comleq{\eqref{eq:ertlowb1}} 1 + \frac{z^2}{2} + \frac{1}{4} \leq \frac{7}{4},
	}
	where
	we used $z \leq \etat < 1$ in the last inequality.
	
	By~\eqref{eq:rootceq}, $\abs{\ra} = \abs{\rb} = \sqrt{\etat}$.
	Thus,
	\eq{
		T_{k}\pr{x}^2 = \abs{b_{k}}^2 \leq 2\abs{\frac{1 - \rb}{\ra - \rb}}^2\abs{\ra}^{2k} + 2\abs{\frac{1 - \ra}{\rb - \ra}}^2\abs{\rb}^{2k} \leq 7\etat^k,\ \forall k\geq 0.
	}
	Since the above arguments hold for any $x\in [0, 1 - \gap]$, ~\eqref{eq:lcal1} follows.

	To prove~\eqref{eq:Chebyssp1}, it suffices to prove
	\eq{
		\nt{\acon\Wt^k\acon\vv_{\ca}}^2 \leq \pe\rw^{2k }\nt{\Con\vv}^2,\ \forall \vv\in \Real^n.
	}
	Denote $\vv^0 = \Con\vv$ and
	\eq{
		\vv^k = \br{\acon\Wt^{k-1}\acon\vv_{\ca}}_{1:n}.
	}
	Then, by the definition of $\Wt$,  the update rule of $\vv^k$ can be written as
	\eql{\label{eq:vkupdt1}}{
		&\vv^0 = \vv^1 = \Con\vv, \\
		&\vv^{k+2} = \pr{1+\ew}\WW\vv^{k+1} - \ew\vv^k.
	}
	We also have
	\eq{
		\acon\Wt^k\acon\vv_{\ca} =
		\begin{pmatrix}
			\vv^{k+1} \\
			\vv^k
		\end{pmatrix}.
	}
	
	Let $0 \leq \la_n \leq \la_{n-1} \leq \cdots \leq \la_2 = 1 - \gap < \la_1 = 1$ be the eigenvalues of $\WW$ ($\la_n \geq 0$ is from the positive semi-definiteness in Assumption~\ref{assp:W}), and let $\dd_i$ be the corresponding eigenvector of $\la_i$.
	Scale each $\dd_i$ such that $\nt{\dd_i} = 1$.
	Then, $\dd_1 = \frac{1}{\sqrt{n}}\one$.
    Since $\one\tp\acon = \zero\tp$, by induction, $\vv^k$ is orthogonal to the all-ones vector for any $k \geq 0$.
    So, we have $\jr{\vv^k, \dd_1} = 0 \ (\forall k\geq 0)$.
	Therefore, for any $k \geq 0$, $\vv^k$ has a unique decomposition as follows:
	\eq{
		\vv^k = \sum_{i=2}^{n} e_i^k \dd_i,
	}
	where $e^k_i = \jr{\vv^k, \dd_i}.  $
	
	Taking inner product with $\dd_i$ on both sides of~\eqref{eq:vkupdt1} yields
	\eq{
		&e^0_i = e^1_i = \jr{\vv^0, \dd_i}, \\
		&e^{k+2}_i = \la_i\pr{1+\ew} e^{k+1}_i - \ew e^k_i,\ k \geq 0.
	}
	Since $\la_i\in [0, 1 - \gap] \ (\forall 2\leq i\leq n)$,
	by~\eqref{eq:lcal1},
	\eq{
		\abs{e^k_i}^2 = \abs{e^0_i\cdot T_k\pr{\la_i}}^2 \leq 7 \abs{e^0_i}^2 \ew^k = 7 \abs{e^0_i}^2 \rw^{2k},\ \forall 2\leq i\leq n.
	}
	Then,
	\eq{
		&\nt{\acon\Wt^k\acon\vv_{\ca}}^2 = \nt{\vv^{k+1}}^2 + \nt{\vv^k}^2 = \sum_{i=2}^{n} \abs{e^{k+1}_i}^2 + \sum_{i=2}^{n} \abs{e^k_i}^2  \\
		\leq& \sum_{i=2}^{n} 7 \abs{e^0_i}^2 \rw^{2\pr{k+1}} + \sum_{i=2}^{n} 7 \abs{e^0_i}^2 \rw^{2{k}} \leq \pe \sum_{i=2}^{n} \abs{e^0_i}^2 \rw^{2k} = \pe \rw^{2k} \nt{\vv^0}^2 = \pe \rw^{2k} \nt{\Con\vv}^2.
	}
\end{proof}
\begin{remark}
	The condition $\etat \in [\frac{1 + \ert}{2}, 1)$ can be relaxed to $\etat \in (\ert, 1)$ with $\pe' = 3 + \frac{2\pr{1 - \ert}}{\etat - \ert}$ and $\pe = 2\pe'$.
	Minimizing the constants $\pe$ and $\pe'$ is out of the scope of this paper.
	
\end{remark}

\subsection{Achieving Chebyshev acceleration without inner loops}
In this section, we show that $\OGT$ achieves the lower bounds on the gradient computation complexity and the communication complexity simultaneously.

We first present a lemma that is derived with the help of Lemma~\ref{lem:lcal1}.
Its meaning can be intuitively understood as follows:
the relations~\eqref{eq:ConG},~\eqref{eq:ConZ},~\eqref{eq:ConU} show that $\E{\mt{\Con\GG^k}^2}$, $\E{\mt{\Con\ZZ^k}^2}$, $\E{\mt{\Con\UU^k}^2}$ are Q-linear sequences with ``additional errors";
while Lemma~\ref{lem:qtorc1} shows that $\E{\mt{\acon\Gt^k}^2}$, $\E{\mt{\acon\Ut^k}^2}$, $\E{\mt{\Con\ZZ^k}^2}$ are R-linear sequences with ``additional errors".
And the ``additional errors" in Lemma~\ref{lem:qtorc1} only differ in constants compared with those in the relations~\eqref{eq:ConG},~\eqref{eq:ConZ},~\eqref{eq:ConU}.
\begin{lemma}\label{lem:qtorc1}
	If
	%\eql{\label{eq:parasqqr1}}{
	$    p \leq \cgap  $,
	%}
	then
	there are sequences of variables $\dr{\cg^k}_{k \geq 0}, \dr{\cu^k}_{k \geq 0}, \dr{\cz^k}_{k \geq 0}$ such that
	for any $k \geq 0$,
	\eql{\label{eq:aconguz1}}{
		\E{\mt{\acon\Gt^k}^2} \leq \cg^k,\ \E{\mt{\acon\Ut^k}^2} \leq \cu^k,\ \E{\mt{\Con\ZZ^k}^2} \leq \cz^k,
	}
	and
	\subeqnum{
		\cg^{k+1} \leq \pr{1 - \cgap}\cg^k  + 2\pe p\E{\mt{{\gF{\XX^k} - \gF{\QQ^k}}}^2}
		%\label{eq:cgrecv1}
		\label{eq:aconG} \\
		\cu^{k+1} \leq    \pr{1 - \cgap}^2\cu^k +  \pe p\E{\mt{\Con\XX^k}^2}
		\label{eq:aconU}  \\
		\cz^{k+1} \leq \pr{1 - \cgap}\cz^k + \E{\frac{4\pe \bet^2}{\cgap}\mt{\Con\XX^k}^2 + \frac{4\pe \eta^2}{\cgap}\cg^k} \notag  \\
		\qquad\qquad + \pe \pr{\frac{\eta^2}{q} + \frac{2\eta^2}{\cgap}}\E{\mt{{\gF{\XX^k} - \gF{\QQ^k  }}}^2  }.
		\label{eq:aconZ}
	}
	
\end{lemma}
\begin{proof}
	From similar arguments for deriving relation~\eqref{eq:ConG}, we have
	\eql{\label{eq:aconGtsupp1}}{
		&\Ek{\mt{\acon\Gt^{k+1}}^2}  \\
		\leq& \pr{1 - p}\mt{\acon\Wt\acon\Gt^k}^2 + p \mt{\acon\Wt\acon\Gt^k + \acon\pr{\cgF{\XX^k} - \cgF{\QQ^k}}}^2 \\
		\leq& \pr{1 + p}\mt{\acon\Wt\acon\Gt^k}^2 + 2p \mt{{\cgF{\XX^k} - \cgF{\QQ^k}}}^2.
	}
	Similar with~\eqref{eq:aconGtsupp1}, we have
	\eq{
		&\pr{1 + p}\mathbb{E}_{k-1}\br{\mt{\acon\Wt\acon\Gt^k}^2}  \\
		\leq& \pr{1 + p}^2\mt{\acon\Wt^2\acon\Gt^{k-1}}^2 + 2p\pr{1 + p}\mt{\acon\Wt\pr{\cgF{\XX^{k-1}} - \cgF{\QQ^{k-1}}}}^2.
	}
	Repeating this process recursively, we obtain
	\eql{\label{eq:aconGt}}{
		&\E{\mt{\acon\Gt^k}^2}  \\
		\leq& \E{\sum_{i=0}^{k-1} 2p \pr{1 + p}^{k-1-i}\mt{\acon\Wt^{k-1-i}\pr{\cgF{\XX^i} - \cgF{\QQ^i}}}^2} \\
    & + \pr{1+p}^k\E{\mt{\acon\Wt^{k}\acon\CG^0}^2}  \\
		\comleq{\eqref{eq:Chebyssp1}}& \E{\sum_{i=0}^{k-1} 2 \pe p \pr{1 + p}^{k-1-i}\pr{1 - \cgap}^{2\pr{k-1-i}}\mt{{\gF{\XX^i} - \gF{\QQ^i}}}^2}  \\
 & + \E{\pe \pr{1+p}^k\pr{1 - \cgap}^{2k}\mt{\acon\GG^0}^2} \\
		\leq& \E{\sum_{i=0}^{k-1} 2\pe p\pr{1 - \cgap}^{{k-1-i}}\mt{{\gF{\XX^i} - \gF{\QQ^i}}}^2 + \pe\pr{1 - \cgap}^{k}\mt{\Con\GG^0}^2}  \\
		\defeq& \cg^k,}
	where the last inequality is from the condition $p \leq \cgap$.
	
	By the definition of $\cg^k$ on the RHS of~\eqref{eq:aconGt}, we have $\E{\mt{\acon\Gt^k}^2} \leq \cg^k $ and
	\eq{
		\cg^{k+1} = \pr{1 - \cgap}\cg^k + 2\pe p\E{\mt{{\gF{\XX^k} - \gF{\QQ^k}}}^2}.
	}

	By the definition of $\xi^k$,
	\eq{
		\Ek{\mt{\acon\Ut^{k+1}}^2} =& \pr{1 - p}\mt{\acon\Wt\acon\Ut^k}^2 + p \mt{\acon\Wt\acon\CX^k}^2  \\
        \leq& \mt{\acon\Wt\acon\Ut^k}^2 + p \mt{\acon\Wt\acon\CX^k}^2.
	}
	Again, we have
	\eq{
		\mathbb{E}_{k-1}\br{\mt{\acon\Wt\acon\Ut^k}^2} \leq \mt{\acon\Wt^2\acon\Ut^{k-1}}^2 + p \mt{\acon\Wt^{2}\acon\CX^k}^2.
	}
	Repeating this process yields
	\eql{\label{eq:defcuk1}}{
		&\E{\mt{\acon\Ut^k}^2}  \\
		\leq& \E{ p \sum_{i=0}^{k-1} \mt{\acon\Wt^{k-i}\acon\CX^i}^2 + \mt{\acon\Wt^k\acon\CU^0}^2  }  \\
		\comleq{\eqref{eq:Chebyssp1}}&  \E{ \pe p \sum_{i=0}^{k-1}\pr{1 - \cgap}^{2\pr{k-i}}\mt{\Con\XX^i}^2 + \pe\pr{1 - \cgap}^{2k}\mt{\Con\UU^0}   }  \\
		\defeq&   \cu^k.
	}
	By the definition of $\cu^k$ on the RHS of~\eqref{eq:defcuk1}, we have $\E{\mt{\acon\Ut^k}^2} \leq \cu^k$ and
	\eq{
		\cu^{k+1} = \pr{1 - \cgap}^2\cu^k + \pe p\pr{1 - \cgap}^2\E{\mt{\Con\XX^k}^2} \leq \pr{1 - \cgap}^2\cu^k + \pe p\E{\mt{\Con\XX^k}^2}.
	}
	Similar to~\eqref{eq:ConYsp1}, we have
	\eql{\label{eq:Ztexpansp1}}{
		&\Ek{\mt{\acon\Wt^{k-1-i}\acon\Zt^{i+1}}^2}  \\
		\leq& \pr{1 + \frac{\cgap}{2}}\pr{\frac{1}{1 - \frac{\cgap}{2}}\mt{\acon\Wt^{k-i}\acon\Zt^i}^2 + \frac{4\bet^2}{\cgap}\mt{\acon\Wt^{k-i}\acon\CX^i}^2 + \frac{4\eta^2}{\cgap}\mt{\acon\Wt^{k-i}  \acon\CG^i}^2}  \\
		& + \pr{\frac{\eta^2}{q} + \frac{2\eta^2}{\cgap}}\mt{\acon\Wt^{k-i}\pr{\cgF{\QQ^i} - \cgF{\XX^i}}}^2,\ \forall 0\leq i < k.
	}
	Using~\eqref{eq:Ztexpansp1} recursively yields
	\eql{\label{eq:defcz1}}{
		&\E{\mt{\acon\Zt^k}^2}  \\
		\leq& \E{ \pr{\frac{1 + \frac{\cgap}{2}}{1 - \frac{\cgap}{2}}}^k\mt{\acon\Wt^k\acon\Zt^0}^2} \\
        & + \E{\sum_{i=0}^{k-1} \pr{1 + \frac{\cgap}{2}}^{k-i}\pr{\frac{4\bet^2}{\cgap}\mt{\acon\Wt^{k-i}\acon\CX^i}^2 + \frac{4\eta^2}{\cgap}\mt{\acon\Wt^{k-i}  \acon\CG^i}^2}}  \\
		& + \E{ \sum_{i=0}^{k-1} \pr{1 + \frac{\cgap}{2}}^{k-i} \pr{\frac{\eta^2}{q} + \frac{2\eta^2}{\cgap}}\mt{\acon\Wt^{k-i}\pr{\cgF{\QQ^i} - \cgF{\XX^i}}}^2     }
		\\
		\comleq{\eqref{eq:Chebyssp1}}&  \E{\sum_{i=0}^{k-1} \pr{1 - \cgap }^{k-i}\pr{\frac{4\pe \bet^2}{\cgap}\mt{\Con\XX^i}^2 + \frac{4\pe \eta^2}{\cgap}\cg^i}}  \\
		& + \pe\E{\sum_{i=0}^{k-1} \pr{1 - \cgap }^{k-i}\pr{\frac{\eta^2}{q} + \frac{2\eta^2}{\cgap}}\mt{{\gF{\QQ^i} - \gF{\XX^i}}}^2}    \\
		& + \pe\pr{1 - \cgap}^k\mt{\Con\ZZ^0}^2  \\
		\defeq& \cz^k.
	}
	By the definition of $\cz^k$ on the RHS of~\eqref{eq:defcz1}, we have $\E{\mt{\Con\ZZ^k}^2} \leq  \E{\mt{\acon\Zt^k}^2} \leq \cz^k$ and
	\eq{
		\cz^{k+1} =& \pr{1 - \cgap}\cz^k  \\
                & + \E{\frac{4\pe\bet^2}{\cgap}\mt{\Con\XX^k}^2 + \frac{4\pe\eta^2}{\cgap}\cg^k
			+ \pe\pr{\frac{\eta^2}{q} + \frac{2\eta^2}{\cgap}}\mt{{\gF{\QQ^k} - \gF{\XX^k}}}^2  }.
	}
	%  \pr{1 + \frac{\gap}{2}}
\end{proof}

The next lemma is derived in a way similar with Lemma~\ref{lem:conbd1}.
It constructs a Lyapunov function to bound the consensus errors.
\begin{lemma}\label{lem:cconbd1}
	If the parameters and stepsize satisfy
	\subeql{\label{eq:cLyparastpun1}}{
		%&  p\leq \gap  \\
		&\frac{2\pr{1 - \tau}}{2 - \tau} + \frac{192\pe  \pr{1 - \tau}\gam^2\bet^2}{\tau\cgap^2} + \frac{\tau}{4} + \frac{16\pe    p}{\cgap} \leq 1 - \frac{\tau}{8}  \\
		%&\frac{17\tau  \gap}{40} \leq \frac{3\gap}{2} - \gap^2  \\
		&\frac{\alp^2}{\pr{1 - \tau}\gam^2} \leq 1  \\
		&\frac{192\pe \pr{1 - \tau}\gam^2\eta^2  L^2}{\tau\cgap} \leq \frac{\cgap}{2},
	}
	then,
	\eql{\label{eq:cLyConX}}{
		\E{\mt{\Con\XX^k}^2 + \mt{\Con\QQ^k}^2} \leq& \pr{1 - \min\dr{\frac{p}{2}, \tau, \frac{\cgap}{2}}}\cLy^k - \cLy^{k+1}  \\
		& + \pe  \pbb\E{\mt{\gF{\XX^k} - \gF{\QQ^k}}^2},
	}
	where the Lyapunov function is defined as
	\eq{
		\cLy^k =& \frac{8}{\tau}\Big( \E{\frac{\tau}{4p}\mt{\Con\QQ^k}^2 + \frac{4\pr{1 - \tau}}{4 - \tau}\mt{\Con\YY^{k}}^2} + \frac{16}{\cgap}\cu^k  \\
            & \qquad + \pr{1 - \frac{\cgap}{6}}\frac{48\pr{1 - \tau}\gam^2}{\tau\cgap}\cz^k
			+ \frac{1}{\cgap  L^2}\cg^k\Big),
	}
	and
	$\pbb$ is defined analogously to~\eqref{eq:defpb1}, i.e.,
	\eq{\pbb = \frac{384\pr{1 - \tau}\gam^2}{\tau^2\cgap}\pr{\frac{\eta^2}{q} + \frac{2\eta^2}{\cgap}} + \frac{16  p}{\tau  \cgap  L^2}. }
	
\end{lemma}
\begin{proof}
	Noting that the condition $p \leq \cgap$ in Lemma~\ref{lem:qtorc1} is implied by~\eqref{eq:cLyparastpun1}, so we can apply the result of Lemma~\ref{lem:qtorc1} here.
	Similar to~\eqref{eq:ConX}, we have
	\eql{\label{eq:aconX}}{
		\emt{\Con\XX^k} &\leq \frac{4\pr{1 - \tau}^2}{4 - \tau}\emt{\Con\YY^k} + \frac{8\alp^2}{\tau}\emt{\Con\ZZ^k} + 8\tau\emt{\Con\UU^k}  \\
		&\comleq{\eqref{eq:aconguz1}} \frac{4\pr{1 - \tau}^2}{4 - \tau}\emt{\Con\YY^k} + \frac{8\alp^2}{\tau}\cz^k + 8\tau\cu^k.
	}
	Similar with~\eqref{eq:ConY},
	\eql{\label{eq:aconY}}{
		\emt{\Con\YY^{k+1}  } &\leq \frac{4 - \tau}{4 - 2\tau}\emt{\Con\XX^k} + \frac{2\pr{4 - \tau}\gam^2}{\tau}\pr{\emt{\Con\ZZ^{k+1}} + \emt{\Con\ZZ^k}}  \\
		&\comleq{\eqref{eq:aconguz1}} \frac{4 - \tau}{4 - 2\tau}\emt{\Con\XX^k} + \frac{2\pr{4 - \tau}\gam^2}{\tau}\pr{\cz^{k+1} + \cz^k  }.
	}
	By the definition of $\xi^k$,
	\eql{\label{eq:aconQ}}{
		\E{\mt{\Con\QQ^{k+1}}} = \pr{1 - p}\emt{\Con\QQ^k} + p\emt{\Con\XX^k}.
	}
	Taking weighted sum on both sides of~\eqref{eq:aconG},~\eqref{eq:aconU},~\eqref{eq:aconZ},~\eqref{eq:aconX},~\eqref{eq:aconY},~\eqref{eq:aconQ}  yields
	\eq{
		&\E{\mt{\Con\XX^k}^2 + \frac{\tau}{4p}\mt{\Con\QQ^{k+1  }}^2 + \frac{4\pr{1 - \tau}}{4 - \tau}\mt{\Con\YY^{k+1}}^2}  \\
        & + \frac{16}{\cgap  }\cu^{k+1} + \frac{48\pr{1 - \tau}\gam^2}{\tau\cgap}\cz^{k+1} + \frac{1 }{\cgap L^2}\cg^{k+1}  \\
		\leq&  \E{\pr{\frac{2\pr{1 - \tau}}{2 - \tau} + \frac{192\pe \pr{1 - \tau}\gam^2\bet^2}{\tau\cgap^2} + \frac{\tau}{4} + \frac{16\pe  p}{\cgap}}\mt{\Con\XX^k}^2}  \\
			& + \E{\pr{1 - p}\frac{\tau}{4p}\mt{\Con\QQ^k}^2
			+ \frac{4\pr{1 - \tau}^2}{4 - \tau}\mt{\Con\YY^k}^2}  \\
		& + \pr{1 - 2\cgap + \cgap^2 + \frac{\tau\cgap}{2}}\frac{16}{\cgap}\cu^k
		+ \pr{1 - \cgap + \frac{\cgap}{6} + \frac{\alp^2\cgap}{6\pr{1 - \tau}\gam^2}}\frac{48\pr{1 - \tau}\gam^2}{\tau\cgap}\cz^k  \\
		& + \frac{8\pr{1 - \tau}\gam^2}{\tau}\cz^{k+1}
		+ \pr{1 - \cgap       + \frac{192\pe \pr{1 - \tau}\gam^2\eta^2  L^2}{\tau\cgap}}\frac{1}{\cgap  L^2}\cg^k  \\
		& + \pe \paa\E{\mt{\gF{\XX^k} - \gF{\QQ^k}}^2}  \\
		\comleq{\eqref{eq:cLyparastpun1}}&  \E{\pr{1 - \frac{\tau}{8}}\mt{\Con\XX^k}^2
			+ \pr{1 - p}\frac{\tau}{4p}\mt{\Con\QQ^k}^2
			+ \frac{4\pr{1 - \tau}^2}{4 - \tau}\E{\mt{\Con\YY^k}^2}}  \\
		& + \pr{1 - \frac{\cgap}{2}}\frac{16    }{\cgap}\cu^k
		 + \pr{1 - \frac{2\cgap}{3}}\frac{48\pr{1 - \tau}\gam^2}{\tau\cgap}\cz^k + \frac{8\pr{1 - \tau}\gam^2}{\tau}\cz^{k+1}  \\
		& + \pr{1 - \frac{\cgap}{2}}\frac{1}{\cgap  L^2}\cg^k
		 + \pe \paa\E{\mt{\gF{\XX^k} - \gF{\QQ^k}}^2},
	}
	where
	\eq{
	    	\paa = \frac{48\pr{1 - \tau}\gam^2}{\tau\cgap}\pr{\frac{\eta^2}{q} + \frac{2\eta^2}{\cgap}} + \frac{2p}{\cgap  L^2}.
	}
	
	By rearranging the above equation, we have
	\eq{
		&  \E{\frac{\tau}{8}\pr{\mt{\Con\XX^k}^2 + \mt{\Con\QQ^{k}}^2}
			+ \frac{\tau}{4p}\mt{\Con\QQ^{k+1}}^2
			+ \frac{4\pr{1 - \tau}}{4 - \tau}\mt{\Con\YY^{k+1}}^2}  \\
		& + \frac{16}{\cgap}\cu^{k+1} + \pr{1 - \frac{\cgap}{6}}\frac{48\pr{1 - \tau}\gam^2}{\tau\cgap}\cz^{k+1} + \frac{1}{\cgap  L^2}\cg^{k+1}  \\
		\leq&   \E{\pr{1 - \frac{p}{2}} \frac{\tau}{4p}\mt{\Con\QQ^k}^2     + \frac{4\pr{1 - \tau}^2}{4 - \tau}\mt{\Con\YY^{k}}^2}  \\
        & + \pr{1 - \frac{\cgap}{2}}\frac{16}{\cgap}\cu^k + \pr{1 - \frac{2\cgap}{3}}\frac{48\pr{1 - \tau}\gam^2}{\tau\cgap}\cz^k  \\
		& + \pr{1 - \frac{\cgap}{2}}\frac{1}{\cgap  L^2}\cg^k + \pe\paa\E{\mt{\gF{\XX^k} - \gF{\QQ^k}}^2}  \\
		\leq& \pr{1 - \min\dr{\frac{p}{2}, \tau, \frac{\cgap}{2}}}
		 \Big(\E{\frac{\tau}{4p}\mt{\Con\QQ^k}^2 + \frac{4\pr{1 - \tau}}{4 - \tau}\mt{\Con\YY^{k}}^2}  \\
        & + \frac{16}{\cgap}\cu^k + \pr{1 - \frac{\cgap  }{6}}\frac{48\pr{1 - \tau}\gam^2}{\tau\cgap}\cz^k
			+ \frac{1}{\cgap  L^2}\cg^k \Big)  \\
		& + \pe \paa\E{\mt{\gF{\XX^k} - \gF{\QQ^k}}^2}.
	}
	%Using the fact that $\pr{1 - \frac{2\gap}{3}} \leq \pr{1 - \frac{\gap}{6}}\pr{1 - \frac{\gap}{2}}$ and
	Then, rearranging the above equation yields ~\eqref{eq:LyConX}.
\end{proof}

Since~\eqref{eq:caupdt1} is exactly the same as~\eqref{eq:aupdt1},  relation~\eqref{eq:diffgF} still holds in this section.
And it is easy to see that the result of Lemma~\ref{lem:cconbd1} has a very similar form with that in Lemma~\ref{lem:conbd1}.
Thus, the following lemma can be derived by almost the same arguments as the proof of Lemma~\ref{lem:ratesupp2}. We omit its proof here.
\begin{lemma}\label{lem:ccratesupp2}
	With $\pc, \pd$ defined in Lemma~\ref{lem:ratesupp2}.
	If the parameters and stepsize satisfy
	\subeqnuml{\label{eq:cparastpdiffgF}}{
		\frac{2\eta L^2}{qn} \leq \frac{\pc}{2},\
		\pe \pc\pbb \leq \frac{\eta}{2qn}
		\label{eq:cparaLyhalf}  \\
		\frac{8\eta L}{q} \leq \frac{\tau}{\alp}, \label{eq:cBGtdba1}
	}
	then we have the following inequality
	\eql{\label{eq:cLydiffgF1}}{
		&\E{\pc\mt{\Con\XX^k}^2
			+ \frac{\eta}{qn}\mt{\gF{\XX^k} - \gF{\QQ^k}}^2}  \\
		\leq& \E{\frac{\tau}{\alp}\pr{f\pr{\ua^k} - f\pr{\xa^k} - \jr{\da^k, \ua^k - \xa^k}}
			+ \pd\pr{\pr{1 - \min\dr{\frac{p}{2}, \tau, \frac{\gap}{2}}}\cLy^k - {\cLy^{k+1}}}}.
	}
	
\end{lemma}

Now, we are ready to show the complexities of $\OGT$.
\begin{theorem}
	Define the Lyapunov function
	\eq{
		\cLya^k = \E{\frac{1}{\gam}\pr{f\pr{\ya^{k}} - f\pr{\xx^*}} + \LyU^{k} + \frac{1+\bet}{2\eta}\nt{\za^{k} - \xx^*}^2 + \pd\cLy^{k}}
	}
    and a constant
    \eq{
        \rate = \min\dr{\frac{\gam}{4}, \pra, \frac{\bet}{1+\bet}, \frac{p}{2}, \tau, \frac{\cgap}{2}},
    }
    where $\pra$ is given by~\eqref{eq:defpra1} and $\cgap$ is given by~\eqref{eq:defjing1}.

	If the parameters and stepsize satisfy~\eqref{eq:cLyparastpun1},~\eqref{eq:cparastpdiffgF} and
	\subeqnum{
		L\gam \leq \frac{1}{4\eta},\ \frac{\bet}{2\eta} \leq \frac{\mu}{4}  \label{eq:celimstrcv1}  \\
		\frac{1}{\gam} - \frac{1 - \alp - \tau}{\alp} - \frac{1}{2} \leq 0 \label{eq:celimfxk1}  \\
		\frac{1 - \alp - \tau}{\alp} \leq \frac{1}{\gam} - \frac{1}{4}, \label{eq:cupbfyk1}
	}
	then $\cLya^k$ converges linearly with
	\eq{
		\cLya^{k+1} \leq \pr{1 - \rate}\cLya^k.
	}
	Specifically, we can choose
	\eql{\label{eq:cparastpex1}}{
		\tau = \frac{1}{2},\
		p = q = \frac{\cgap  }{60750},\
		\alp = \frac{1}{45\sqrt{\ka}},\
		\gam = \frac{4\alp}{4 - 4\tau - 3\alp},\
		\eta = \frac{\cgap\sqrt{\ka}}{50000  L},\
		\bet = \frac{\mu\eta}{2}.
	}
	In this case, to have $\cLya^K \leq \eps\cLya^0$,
	the gradient computation complexity is $O\pr{\sqrt{\ka}\log\frac{1}{\eps}}$
	and the communication complexity is $O\pr{\sqrt{\frac{\ka}{\gap}}\log\frac{1}{\eps} }$.
	
\end{theorem}
\begin{proof}
	Since~\eqref{eq:caupdt1} is exactly the same with~\eqref{eq:aupdt1}, relation~\eqref{eq:LyU} and Lemma~\ref{lem:dajrsupp} still hold here. And noting that the result of Lemma~\ref{lem:ccratesupp2} has almost the same form as that in Lemma~\ref{lem:ratesupp2}, we can derive the following inequality by almost the same arguments as those for deriving~\eqref{eq:subst1}.
	
	\eq{
		&\E{f\pr{\xa^k} - f\pr{\xx^*}} \\
		\leq& \E{\frac{1 - \alp - \tau}{\alp}\pr{f\pr{\ya^k} - f\pr{\xa^k}}  + \frac{1}{\gam}\pr{f\pr{\xa^k} - f\pr{\ya^{k+1}}} + \frac{1}{2}\pr{f\pr{\xa^k} - f\pr{\xx^*}}}  \\
		& + \E{\pr{1 - \pra}\LyU^k - \LyU^{k+1}}
		+ \E{\frac{1}{2\eta}\nt{\za^k - \xx^*}^2 - \frac{1 + \bet}{2\eta}{\nt{\za^{k+1} - \xx^*}^2}}  \\
		& + \pd\pr{\pr{1 - \min\dr{\frac{p}{2}, \tau, \frac{\cgap}{2}}}\E{\cLy^k - {\cLy^{k+1}}}}.
	}
	By rearranging the above inequality, we have
	\eq{
		&\cLya^{k+1} = \E{\frac{1}{\gam}\pr{f\pr{\ya^{k+1}} - f\pr{\xx^*}} + \LyU^{k+1} + \frac{1+\bet}{2\eta}\nt{\za^{k+1} - \xx^*}^2 + \pd\cLy^{k+1}    }  \\
		\leq& \E{\frac{1 - \alp - \tau}{\alp}\pr{f\pr{\ya^k} - f\pr{\xx^*}} + \pr{\frac{1}{\gam} - \frac{1 - \alp - \tau}{\alp} - \frac{1}{2}}\pr{f\pr{\xa^k} - f\pr{\xx^*}}}  \\
        & + \E{\pr{1 - \pra}\LyU^k + \frac{1}{2\eta}\nt{\za^k - \xx^*}^2 + \pd\pr{1 - \min\dr{\frac{p}{2}, \tau, \frac{\cgap}{2}}}\cLy^k}  \\
		\comleq{\eqref{eq:celimfxk1},\eqref{eq:cupbfyk1}}& \E{\pr{\frac{1}{\gam} - \frac{1}{4}}\pr{f\pr{\ya^k} - f\pr{\xx^*}} + \pr{1 - \pra}\LyU^k
			+ \frac{1}{2\eta}\nt{\za^k - \xx^*}^2} \\
            & + \pd\pr{1 - \min\dr{\frac{p}{2}, \tau, \frac{\cgap}{2}}}\E{\cLy^k}  \\
		\leq& \pr{1 - \rate}\E{{\frac{1}{\gam}\pr{f\pr{\ya^{k}} - f\pr{\xx^*}} + \LyU^{k} + \frac{1+\bet}{2\eta}\nt{\za^{k} - \xx^*}^2 + \pd\cLy^{k}}} \\
		=& \pr{1 - \rate}\cLya^k.
	}
	
	When we choose the parameters and stepsize as in~\eqref{eq:cparastpex1}, $\rate = O\pr{\frac{\cgap}{\sqrt{\ka}}}$.
	By the definition of $\cgap$ in~\eqref{eq:defjing1}, $\cgap = O\pr{\sqrt{\gap}}$.
	Then, the communication complexity is $O\pr{\frac{\sqrt{\ka}}{\cgap}\log\frac{1}{\eps}} = O\pr{\sqrt{\frac{\ka}{\gap}}\log\frac{1}{\eps}}$.
	For each communication round, in expectation, at most ${p + q} = O\pr{\cgap}$ gradient computations are implemented.
	Therefore, the gradient computation complexity is $O\pr{\frac{\sqrt{\ka}}{\cgap}\log\frac{1}{\eps}} \cdot O\pr{\cgap} = O\pr{\sqrt{\ka}\log\frac{1}{\eps}}$.
\end{proof}

\begin{remark}
	Following a similar way as what we did in Section~\ref{sec:OGT1}, many previous GT-based methods, including the classical gradient tracking, acc-DNGD-SC~\cite{qu2019accelerated}, and Acc-GT~\cite{li2021accelerated},  can be combined with \lCA\ and have better dependence on the graph condition number $\frac{1}{\gap}$.
	We omit the detailed discussion here.
	
\end{remark}

\newcommand\degree{{\rm deg}}
\def\WWt#1{\WW^{\pr{#1 }}}
\section{Numerical Experiments}\label{sec:numerical1}
In this section, we verify the numerical efficiency of $\OGT$ by comparing it with two single-loop accelerated methods APAPC~\cite{kovalev2020optimal} and Acc-GT~\cite{li2021accelerated} and a double-loop and optimal algorithm OPAPC~\cite{kovalev2020optimal} on the following $\ell^2$-penalized logistic regression problem:
\eq{
    f\pr{\xx} = \sum_{i\in \calN} f_i\pr{\xx} = \frac{1}{n}\sum_{i\in \calN} \pr{\log\pr{1 + \exp\pr{-y_i\zz_i\tp\xx}} + \frac{\mu\nt{\xx}^2}{2}},
}
where $\zz_i\in \Real^4$ is the feature vector and $y_i\in \dr{-1, +1}$ is the label.
Each private objective function $f_i\pr{\xx}$ is $\mu$-strongly convex.
Here, we set $n = 200$ and $\mu = 0.01$.
The training vector of each agent $i$ is chosen from
Banknote Authentication Data Set in UCI Machine Learning Repository~\cite{misc_banknote_authentication_267}\footnote{https://archive.ics.uci.edu/ml/datasets/banknote+authentication}
randomly without replacement.

We test the performance of the algorithms on $2$ different communication networks.
The first one is an $n$-cycle whose edge set is given by $\calE = \dr{\pr{i, i+1}: 1\leq i\leq n-1}\cup \dr{\pr{n, 1}}$.
The gossip matrix $\WWt{1}$  is
\eq{
    \WWt{1}_{ij} =
    \left\{
    \begin{split}
        & \ \frac{1}{4},\quad \pr{i, j}\in \calE \\
        & \ \frac{1}{2},\quad i = j \\
        & \ 0,\quad \text{otherwise }
    \end{split}
    \right.
}
We consider such a gossip matrix since it has a small spectral gap $\gap \approx 0.00025 $ which can challenge the performance of the algorithms in the dependence on the graph condition number $\frac{1}{\gap}$.

The second network is a denser graph generated by adding $50$ edges to the $n$-cycle randomly.
Denote its edge set by $\calE^+$.
The degree of vertex $i$ in this denser graph is denoted by $\degree\pr{i}$.
The corresponding gossip matrix $\WWt{2}$ is generated from a ``lazy" version of the Metropolis weights:
\begin{alignat*}{2}
    \WWt{2}_{ij} =
    \begin{cases}
         \ \frac{1}{2\max\dr{\degree\pr{i}, \degree\pr{j}}},  \\
         \ 1 - \sum_{k\in [n]\backslash \dr{i} } \WWt{2}_{ik}, \\
         \ 0,  \\
    \end{cases}
    & \quad       &
    \begin{aligned}
        &  \pr{i, j}\in \calE^+ \\
        &  i = j  \\
        &  \text{otherwise }
    \end{aligned}
\end{alignat*}
The gossip matrix $\WWt{2}$ has a much larger spectral gap $\gap \approx 0.009 $ than $\WWt{1}$ due to the better graph connectivity.

%Note that compared with APAPC and Acc-GT, $\OGT$ has better dependence on the network condition number $\frac{1}{\gap}$. So, the gossip matrix with a smaller spectral gap may show the superiority of $\OGT$ more easily.
The parameters and the stepsize for $\OGT$ are chosen as follows:
(i) for the $n$-cycle, we set
$\alp = 0.02$, $\tau = 0.1$, $\gam = \frac{4\alp}{4 - 4\tau - 3\alp}$, $\eta = 0.05$, $\bet = \frac{\eta\mu}{2}$, $p = q = 0.1$ with the value $\ew \approx 0.978 $ computed from~\eqref{eq:defjing1};
(ii) for the denser graph, we set $\alp = 0.02$, $\tau = 0.1$, $\gam = \frac{4\alp}{4 - 4\tau - 3\alp}$, $\eta = 0.1$, $\bet = \frac{\eta\mu}{2}$, $p = q = 0.2$ and $\ew \approx 0.883 $ by~\eqref{eq:defjing1}.
We always let $\xi^k = q\zeta^k$ in both networks.

{
The parameters for APAPC are tuned to optimze its performance.
Specifically, the parameters are chosen as follows: (i) for the $n$-cycle, $\tau = 0.2$, $\eta = 0.02$, $\theta = 40$; (ii) for the denser graph, $\tau = 0.1$, $\eta = 0.11$, $\theta = 10$.
For Acc-GT, we take $\thek = \frac{\sqrt{\mu\alp}}{2}$ as proposed in~\cite{li2021accelerated}.
Then, we hand-tune the value of $\alp$ to optimize its performance with $\alp=0.0001$ for the $n$-cycle and $\alp = 0.0004$ for the denser graph.
    For OPAPC, the number of iterations in the inner loops of Chebyshev acceleration are chosen as $T_0 = \floor{1/\sqrt{\gap}}$ as recommended by the classical implementation of Chebyshev acceleration~\cite{scaman2017optimal}.
    Specifically, for the $n$-cycle, $T_0 = 63$; for the denser graph, $T_0 = 10$.
    We hand-tune the other parameters to optimize the performance of OPAPC.
    The parameters are as follows:
    (i) for the $n$-cycle, $\tau = 0.05$, $\eta = 0.3$, $\theta = 3$;
    (ii) for the denser graph, $\tau = 0.03$, $\eta = 0.3$, $\theta = 3$.
}

The performance of the different algorithms is illustrated in Fig.~\ref{fig:cp1},
where the $y$-axis represents the loss $\frac{1}{n}\sum_{i\in \calN} f\pr{\xx^K_i} - f\pr{\xx^*}$, and the $x$-axis represents the communication round (iteration number) or the number of gradient computations.
Since $\xx^*$ is unknown, we estimate $f\pr{\xx^*}$ by running the classical Nesterov's accelerated gradient descent for $20000$ iterations.

\begin{figure}[h]
\centering
\includegraphics[scale=0.9]{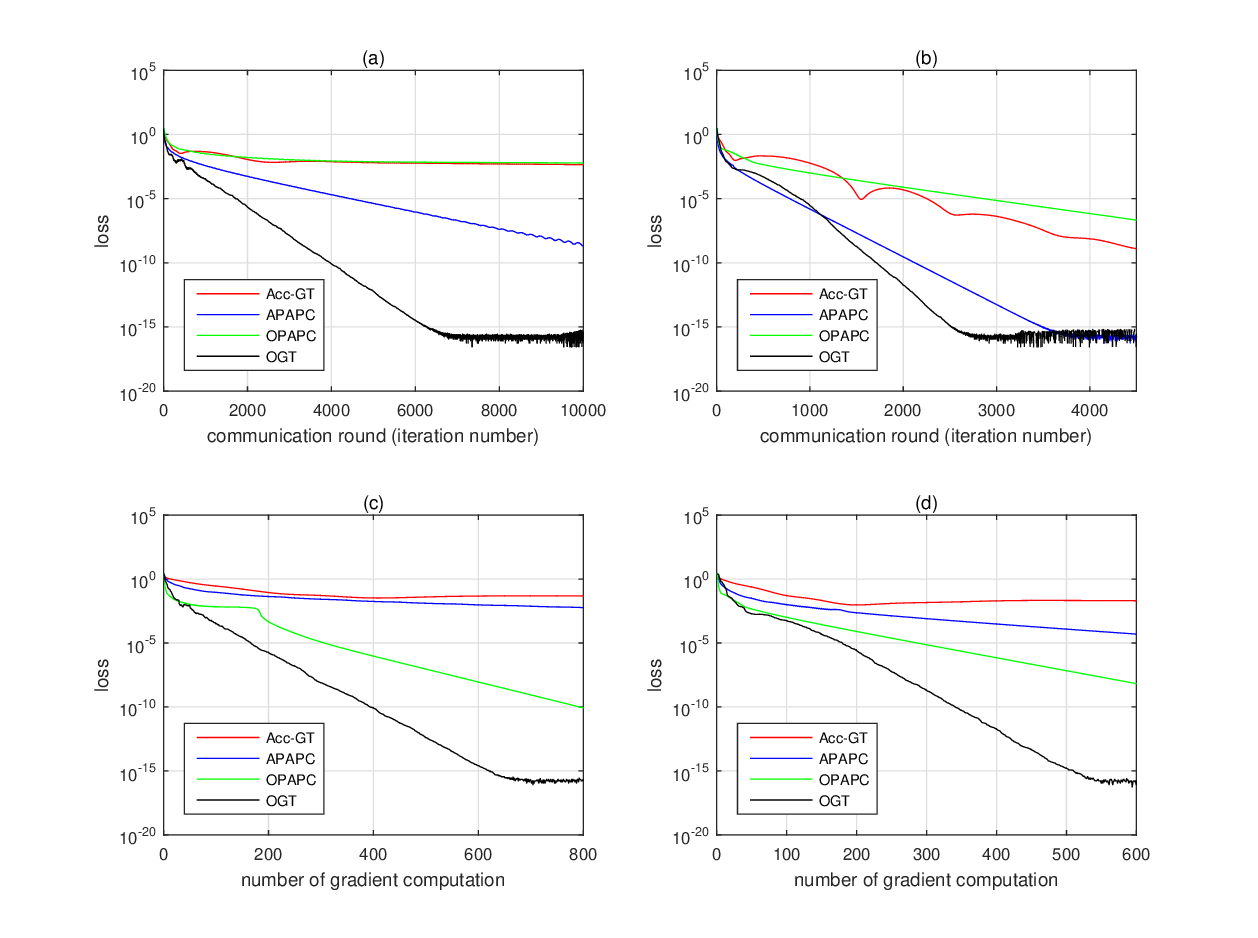}
\caption{Comparison of the performance of different algorithms with respect to the communication rounds (top) and gradient computation (bottom) on the $n$-cycle (left) and the denser graph (right).}\label{fig:cp1}
\centering
\end{figure}

We have some observations and remarks below:
\begin{itemize}
  \item $\OGT$ converges fast than APAPC and Acc-GT with respect to the communication round in both networks.
    This coincides with the better communication complexity of $\OGT$.
 When comparing Fig~\ref{fig:cp1}~(a)(b),
      the ratios of the convergence rates between $\OGT$ and the other methods are much larger for the $n$-cycle communication network than for the denser graph.
      This is expected since $\OGT$, APAPC, and Acc-GT all have $O\pr{\sqrt{\ka}}$ dependence on  the function condition number in the communication complexity, while $\OGT$ has better dependence on the graph condition number.
      Thus, the advantage of $\OGT$ becomes greater when the graph condition number is larger.

  \item The performance of $\OGT$ is much better than Acc-GT and APAPC when comparing with respect to the number of gradient computations.
      When $\OGT$ converges to a highly accurate solution with loss around $10^{-15}$, the losses of the other methods only slightly decrease.
      This is because when $p, q$ are set to be small, most of the iterations of $\OGT$ do not involve gradient computations.
      By comparison, APAPC and Acc-GT require one gradient computation at each iteration.

  {
  \item The convergence rate of OGT is much better than that of OPAPC, though they are both provably optimal.
  This is expected since OPAPC has an inner loop of Chebyshev acceleration while OGT is a single-loop algorithm.
  In addition, the gradient computation cost of OGT is also less than that of OPAPC.
  It is also worth noting that the gradient computation costs of OGT to reach an accurate solution with loss around $10^{-15}$ vary little between the $n$-cycle and the denser graph.
  This is expected since OGT has gradient computation complexity $O\pr{\sqrt{\ka}\log\frac{1}{\eps}}$ which remains invariant on different networks.
  }

\end{itemize}

\section{Conclusion}\label{sec:conclude1}
In this paper, we propose a novel gradient tracking method termed the \emph{\OGTnm}\ (\OGT) method for decentralized optimization.
	To our knowledge, $\OGT$ is the first single-loop method  that is optimal in both the gradient computation and the communication complexities in the class of gradient-type methods.
	To develop \OGT\ , we first design the \emph{\SSGTnm}\ (\SSGT) method which  has a very different scheme compared with previous GT-based methods.
 %and outperforms previous single-loop GT-based methods.
 %has the potential to be extended to more general settings. % including directed graphs, time-varying graphs and so on.
	Then, we develop the \lCA\ (LCA) technique to accelerate $\SSGT$, which leads to $\OGT$.
	The LCA technique can also be used to accelerate many other GT-based methods with respect to the graph condition number.

\bibliography{ref}

\input{Appendix}

%\bibliographystyle{alpha}
%\bibliography{Ref}

\end{document}

%% file: newcommand.tex
\def\eq#1{\begin{equation*}\begin{split}#1\end{split}\end{equation*}}
\def\eql#1#2{\begin{equation}{#1}\begin{split}#2\end{split}\end{equation}}

\def\subeql#1#2{\begin{equation}{#1}\left\{\begin{split}#2\end{split}\right.\end{equation}}
\def\subeqnum#1{\begin{subequations}\begin{numcases}{}#1\end{numcases}\end{subequations}}
\def\subeqnuml#1#2{\begin{subequations}{#1}\begin{numcases}{}#2\end{numcases}\end{subequations}}

\def\comeq#1{\stackrel{\mathrm{#1}}{=}}
\def\comleq#1{\stackrel{\mathrm{#1}}{\leq}}

\def\expec#1{{\mathbb{E}}\left[ #1 \right]}

\newcommand{\E}{\mbox{{\bf E}}}

\def\defeq{\stackrel{\mathrm{def}}{=}}

\def\pr#1{\left( #1 \right ) }
\def\br#1{\left[ #1 \right ] }
\def\dr#1{\left\{#1\right\}}
\def\jr#1{\left< #1 \right>  }

\def\floor#1{\left\lfloor #1 \right\rfloor}

\def\abs#1{\left\lvert #1 \right\rvert }

\def\calC{\mathcal{C}}

\def\calE{\mathcal{E}}
\def\calG{\mathcal{G}}

\def\calN{\mathcal{N}}
\def\calF{\mathcal{F}}

\newcommand\PPi{\boldsymbol{\Pi}}

\def\aa{\pmb{\mathit{a}}}
\newcommand\bb{\boldsymbol{\mathit{b}}}

\newcommand\dd{\boldsymbol{\mathit{d}}}

\renewcommand\gg{\boldsymbol{\mathit{g}}}

\newcommand\qq{\boldsymbol{\mathit{q}}}

\renewcommand\ss{\boldsymbol{\mathit{s}}}

\newcommand\uu{\boldsymbol{\mathit{u}}}
\newcommand\vv{\boldsymbol{\mathit{v}}}

\newcommand\yy{\boldsymbol{\mathit{y}}}
\newcommand\zz{\boldsymbol{\mathit{z}}}
\newcommand\xx{\boldsymbol{\mathit{x}}}

\renewcommand\AA{\boldsymbol{\mathit{A}}}
\newcommand\BB{\boldsymbol{\mathit{B}}}
\newcommand\CC{\boldsymbol{\mathit{C}}}
\newcommand\DD{\boldsymbol{\mathit{D}}}

\newcommand\FF{\boldsymbol{\mathit{F}}}
\newcommand\GG{\boldsymbol{\mathit{G}}}
\newcommand\II{\boldsymbol{\mathit{I}}}

\newcommand\MM{\boldsymbol{\mathit{M}}}

\newcommand\QQ{\boldsymbol{\mathit{Q}}}

\renewcommand\SS{\boldsymbol{\mathit{S}}}

\newcommand\UU{\boldsymbol{\mathit{U}}}
\newcommand\WW{\boldsymbol{\mathit{W}}}

\newcommand\XX{\boldsymbol{\mathit{X}}}
\newcommand\YY{\boldsymbol{\mathit{Y}}}
\newcommand\ZZ{\boldsymbol{\mathit{Z}}}

\def\nm#1{\| #1 \|  }

\def\nt#1{\left\| #1 \right\|}

\def\nf#1{\left\|#1\right\|_{\rm F}}

\newcommand{\mm}[2]{{\left\vert\kern-0.25ex\left\vert\kern-0.25ex\left\vert #2
    \right\vert\kern-0.25ex\right\vert\kern-0.25ex\right\vert}_{#1}}
\def\mt#1{\nf{#1}}
\def\nb#1{\left\|#1\right\|_{2}}

\def\la{\lambda}

\def\tp{^\top}

\newcommand \eps{\epsilon}

\newcommand{\zero}{\mathbf{0}}
\newcommand{\one}{\mathbf{1}}

\newcommand \inv{^{-1}}

\def\MatSize#1#2{\mathbb{R}^{#1\times#2}}
\newcommand \Real{\mathbb{R}}

\newcommand\na{\nabla}

\newcommand*\samethanks[1][\value{footnote}]{\footnotemark[#1]}

\def\gF#1{\na \FF\pr{#1}}
\newcommand\alp{\alpha}
\newcommand\bet{\beta}
\newcommand\gam{\gamma}
\newcommand\xa{\overline{\xx}}
\newcommand\ya{\overline{\yy}}
\newcommand\za{\overline{\zz}}
\newcommand\ua{\overline{\uu}}
\newcommand\ga{\overline{\gg}}
\newcommand\qa{\overline{\qq}}

\newcommand\Con{\PPi}
\newcommand\gap{\theta}
\def\E#1{\expec{#1}}
\def\Ek#1{\mathbb{E}_k\br{#1  }}
\newcommand\Ly{\Phi_{\rm C}}
\newcommand\pa{c_1\pr{\theta, \gam, \tau, p, q, \eta}}
\newcommand\pb{c_2\pr{\theta, \gam, \tau, p, q, \eta}}
\newcommand\paa{c_1(\cgap, \gam, \tau, p, q, \eta)}
\newcommand\pbb{c_2(\cgap, \gam, \tau, p, q, \eta)}
\newcommand\pc{c_3\pr{\alp, \gam, \tau}}
\newcommand\pd{c_4}
\newcommand\pe{c_5}
\newcommand\da{\overline{\dd}}
\newcommand\LyU{\Phi_{\rm U}}
\newcommand\pra{\widehat{p}}
\newcommand\rate{\delta}
\newcommand\Lya{\Phi}
\newcommand\ka{\kappa}
\newcommand\Wt{\widetilde{\WW}}
\newcommand\ca{\#}
\newcommand\ew{\widetilde{\eta}_{\rm w}}
\newcommand\rw{\widetilde{\rho}_{\rm w}}
\newcommand\cgap{\widetilde{\gap}_{\rm w}}
\newcommand\acon{\widetilde{\Con}  }
\def\ub#1{\br{#1}_{1:n, :}}
\newcommand\ra{a^{(1)}}
\newcommand\rb{a^{(2)}}
\def\cgF#1{\gF{#1}_{\ca}}
\newcommand\Zt{\widetilde{\ZZ}}
\newcommand\Ut{\widetilde{\UU}}
\newcommand\Gt{\widetilde{\GG}}

\newcommand\CX{\XX_{\ca}}

\newcommand\CU{\UU_{\ca}}
\newcommand\CG{\GG_{\ca}}

\newcommand\cu{\calC_{\rm U}}
\newcommand\cg{\calC_{\rm G}}
\newcommand\cz{\calC_{\rm Z}}
\def\emt#1{\E{\mt{#1}^2}}
\newcommand\cLy{\widetilde{\Phi}_{\rm C}}
\newcommand\cLya{\widetilde{\Phi}}
\newcommand\zat{\widetilde{\zz}}
\newcommand\uat{\widetilde{\uu}}
\newcommand\gat{\widetilde{\gg}}

\newcommand\ztil{\widetilde{\zz}}

\newcommand\catype{\emph{``\ca"-type}}
\newcommand\ert{\eta_{\rm root}}
\newcommand\lCAL{\emph{``loopless Chebyshev acceleration lemma"}}
\newcommand\LCAL{\emph{``Loopless Chebyshev acceleration lemma"}}
\newcommand\lCA{\emph{Loopless Chebyshev Acceleration}}

\newcommand\etat{\widetilde{\eta}}

\newcommand\agt{\overline{\ss}}

\newcommand{\xmark}{\ding{55}}

\newcommand\SSGT{\textsc{SS-GT}}
\newcommand\OGT{\textsc{OGT}}
\newcommand\SSGTnm{``Snapshot" Gradient Tracking}
\newcommand\OGTnm{Optimal Gradient Tracking}

\def\que#1{\emph{\textbf{Question}~\ref{que:#1}}}

%% file: abstract.tex
\begin{abstract}
  In this paper, we focus on solving the decentralized optimization problem of minimizing the sum of $n$ objective functions over a multi-agent network. The agents are embedded in an undirected graph where they can only send/receive information directly to/from their immediate neighbors.
  Assuming smooth and strongly convex objective functions, we propose an \emph{Optimal Gradient Tracking} (\OGT) method that achieves the optimal gradient computation complexity $O\pr{\sqrt{\ka}\log\frac{1}{\eps}}$ and the optimal communication complexity $O\pr{\sqrt{\frac{\ka}{\gap}}\log\frac{1}{\eps}}$ simultaneously, where $\ka$ and $\frac{1}{\gap}$ denote the condition numbers related to the objective functions and the communication graph, respectively.
  To our best knowledge, $\OGT$ is the first single-loop decentralized gradient-type method that is optimal in both gradient computation and communication complexities.
  The development of $\OGT$ involves two building blocks that are also of independent interest.
  The first one is another new decentralized gradient tracking method termed \emph{\SSGTnm}\ (\SSGT), which achieves the gradient computation and communication complexities of $O\pr{\sqrt{\ka}\log\frac{1}{\eps}}$ and $O\pr{\frac{\sqrt{\ka}}{\gap}\log\frac{1}{\eps}}$, respectively. $\SSGT$  can be potentially extended to more general settings compared to \OGT.
  The second one is a technique termed \lCA\ (LCA), which can be implemented ``looplessly" but achieves a similar effect by adding multiple inner loops of Chebyshev acceleration in the algorithm. In addition to $\SSGT$,
  this LCA technique can accelerate many other gradient tracking based methods with respect to the graph condition number $\frac{1}{\gap}$.  

\end{abstract}

%% file: Appendix.tex
\newpage
\begin{appendices}

\section{Appendix}
\subsection{Implementation-friendly versions of $\SSGT$ and $\OGT$}\label{sec:ifDOGTandUOGT}
\begin{algorithm}
	\caption{ Implementation-friendly $\SSGT$    }
	\label{alg:DiOGT}
	
	\KwIn{parameters: $\alp, \bet, \gam, \tau \in (0, 1)$;
		probability: $p, q \in (0, 1)$;
		stepsize: $\eta > 0$;
		initial position: $\XX^0$;
		gossip matrix: $\WW$
	}
	
	\KwOut{$\XX^k$, $\YY^K$ or $\ZZ^K$    }
	
	Initialize $\YY^0 = \ZZ^0 = \UU^0  = \XX^0$, $\GG^0 = \MM^0 = \gF{\XX^0}$. \;
	
	\For{$k = 0$ to $K - 1$}{
%		Sample $\xi^k \sim
%		\begin{pmatrix}
%			0 & 1 \\
%			1 - p & p
%		\end{pmatrix}
%		$
%		and
%		$
%		\zeta^k \sim
%		\begin{pmatrix}
%			0 & \frac{1}{q} \\
%			1 - q & q
%		\end{pmatrix}
%		$. \;
        Sample $\xi^k\sim Bernoulli\pr{p}$, $\zeta^k\sim Bernoulli\pr{q}/q$. \;
		$\XX^{k} = \pr{1 - \alp- \tau}\YY^k + \alp\ZZ^k + \tau\UU^k.  $ \;
		\uIf{$\zeta^k == 0$}{
			$\ZZ^{k+1} = \pr{1 + \bet}\inv\WW\pr{\ZZ^k + \bet\XX^k - \eta\GG^k}.  $
		}\Else{
			$\ZZ^{k+1} = \pr{1 + \bet}\inv\WW\pr{\ZZ^k + \bet\XX^k - \eta\GG^k + \frac{\eta}{q}\pr{\MM^k - \gF{\XX^k}}}.    $
		}
		$\YY^{k+1} = \XX^k + \gam\pr{\ZZ^{k+1} - \ZZ^k}.  $ \;
		\uIf{$\xi^k == 0$}{
			\quad$\MM^{k+1} = \MM^k.  $  \;
			\quad$\UU^{k+1} = \WW\UU^k.  $ \;
			\quad$\GG^{k+1} = \WW\GG^k.  $
		}\Else{
			\quad$\MM^{k+1} = \gF{\XX^k}.  $ \;
			\quad$\UU^{k+1} = \WW\XX^k. $  \;
			\quad$\GG^{k+1} = \WW\GG^k + \MM^{k+1} - \MM^k.  $ \;
		}
	}
	Return $\XX^k$, $\YY^K$ or $\ZZ^K$.
\end{algorithm}

\newpage
\begin{algorithm}
	%\caption{\OGTnm\ (\OGT)  }
    \caption{Implementation-friendly $\OGT$  }
	\label{alg:UOGT}
	
	\KwIn{parameters: $\alp, \bet, \gam, \tau \in (0, 1)$;
		probability: $p, q\in (0, 1)$;
		stepsizes: $\eta > 0, \ew \in (0, 1)$;
		initial position: $\XX^0$;
		gossip matrix: $\WW$
	}
	
	\KwOut{$\XX^k$, $\YY^K$ or $\ub{\Zt^K}$  }
	
	Initialize: $\YY^0 = \XX^0 \in \MatSize{n}{d}$, $\Zt^0 = \Ut^0 = \CX^0 \in \MatSize{2n}{d}$, $\Gt^0 = \cgF{\XX^0} \in \MatSize{2n}{d}$, $\MM^0 = \gF{\XX^0} \in \MatSize{n}{d}$. \;
	\For{$k = 0$ to $K-1$  }{
%		Sample $\xi^k \sim
%		\begin{pmatrix}
%			0 & 1 \\
%			1 - p & p
%		\end{pmatrix}
%		$
%		and
%		$
%		\zeta^k \sim
%		\begin{pmatrix}
%			0 & \frac{1}{q} \\
%			1 - q & q
%		\end{pmatrix}
%		$. \;
        Sample $\xi^k\sim Bernoulli\pr{p}$, $\zeta^k\sim Bernoulli\pr{q}/q$. \;
		%$\XX^k = \pr{1 - \alp - \tau}\YY^k + \alp\ZZ^k + \tau\UU^k.  $ \;
		$\XX^k = \pr{1 - \alp - \tau}\YY^k + \alp\ub{\Zt^k} + \tau\ub{\Ut^k}.  $ \;
		\uIf{$\zeta^k == 0$}{
			$\Zt^{k+1} = \pr{1 + \bet}\inv\Wt\pr{\Zt^k + \bet\CX^k - \eta\CG^k}.  $
		}\Else{
			$\Zt^{k+1} = \pr{1 + \bet}\inv\Wt\pr{\Zt^k + \bet\CX^k - \eta\CG^k + \frac{\eta}{q  }\pr{\MM_{\ca}^k - \cgF{\XX^k}}}.  $
		}
		$\YY^{k+1} = \XX^k + \gam\pr{\ub{\Zt^{k+1}} - \ub{\Zt^k}}.  $  \;
		\uIf{$\xi^k == 0$}{
			$\MM^{k+1} = \MM^k.  $  \;
			$\Ut^{k+1} = \Wt\Ut^k.  $ \;
			$\Gt^{k+1} = \Wt\Gt^k.  $
		}\Else{
			$\MM^{k+1} = \gF{\XX^k}.  $  \;
			$\Ut^{k+1} = \Wt\CX^k.  $  \;
			$\Gt^{k+1} = \Wt\Gt^k + \MM_{\ca}^{k+1} - \MM_{\ca}^k.  $
		}
	}
	Return $\XX^k$, $\YY^K$ or $\ub{\Zt^K}$
	
\end{algorithm}

\newpage
\subsection{Proof of Lemma~\ref{lem:inextgt1}}\label{sec:prfleminextgt1}
  For any $i\in\calN$, in light of the convexity of $f_i$,  we have
  \eq{
    f_i\pr{\xx^k_i} \leq f_i\pr{\bb} + \jr{\na f_i\pr{\xx^k_i}, \xx^k_i - \bb}.
  }
  By the $L$-smoothness of $f_i$, we have
  \eq{
    f_i\pr{\aa} \leq f_i\pr{\xx^k_i} + \jr{\na f_i\pr{\xx^k_i}, \aa - \xx^k_i} + \frac{L}{2}\nt{\aa - \xx^k_i}^2.
  }
  Summing the above two equations and taking average over $i\in\calN$ yield
  \eq{
    f\pr{\aa} &\leq f\pr{\bb} + \jr{\da^k, \aa - \bb} + \frac{L}{2n}\mt{\XX^k - \one\aa}^2  \\
    & \leq f\pr{\bb} + \jr{\da^k, \aa - \bb} + \frac{L}{n}\mt{\XX^k - \one\xa^k}^2 + \frac{L}{n}\mt{\one\xa^k - \one\aa}^2  \\
    & = f\pr{\bb} + \jr{\da^k, \aa - \bb} + \frac{L}{n}\mt{\Con\XX^k}^2 + L\nt{\xa^k - \aa}^2.
  }
%  \eq{
%    f\pr{\aa} \leq f\pr{\bb} + \jr{\da^k, \aa - \bb} + \frac{L}{2n}\mt{\XX^k - \one\aa}.
%  }
  From the $\mu$-strong convexity of each $f_i$, we have
  \eq{
    f_i\pr{\xx^k_i} \leq f_i\pr{\xx^*} + \jr{\na f_i\pr{\xx^k_i}, \xx^k_i - \xx^*} - \frac{\mu}{2}\nt{\xx^k_i - \xx^*}^2.
  }
  In addition, by the $L$-smoothness of $f_i$, we obtain
  \eq{
    f_i\pr{\xa^k} \leq f_i\pr{\xx^k_i} + \jr{\na f_i\pr{\xx^k_i}, \xa^k - \xx^k_i} + \frac{L}{2}\nt{\xa^k - \xx^k_i}^2.
  }
  Summing the above two equations and taking average over $i\in\calN$ leads to
  \eq{
    f\pr{\xa^k} \leq& f\pr{\xx^*} + \jr{\da^k, \xa^k - \xx^*} - \frac{\mu}{2n}\mt{\XX^k - \one\xx^*}^2 + \frac{L}{2n}\mt{\XX^k - \one\xa^k}^2  \\
    \leq& f\pr{\xx^*} + \jr{\da^k, \xa^k - \xx^*} - \frac{\mu}{4n}\mt{\one\xa^k - \one\xx^*}^2 + \frac{\mu + L}{2n}\mt{\XX^k - \one\xa^k}^2  \\
    \leq& f\pr{\xx^*} + \jr{\da^k, \xa^k - \xx^*} - \frac{\mu}{4}\nt{\xa^k - \xx^*}^2 + \frac{L}{n}\mt{\Con\XX^k}^2,
  }
  where the second inequality comes from the elementary inequality: $-\nt{\aa+\bb}^2 \leq -\frac{1}{2}\nt{\aa}^2 + \nt{\bb}^2$, and the last inequality is from the fact that $\mu \leq L$.\qed

\subsection{Proof of Lemma~\ref{lem:gakm}}\label{sec:prfgakm1}
  Firstly, for $k = 0$, $\ga^0 = \frac{1}{n}\one\tp\gF{\QQ^0}. $
  Suppose we have shown relation~\eqref{eq:gakm} for $0, \cdots, k$, next we prove the equality for $k+1$.

  If $\xi^k = 0$, then $\QQ^{k+1} = \QQ^k. $
  By induction hypothesis and the fact $\one\tp\WW = \one\tp$,
  \eq{
    \ga^{k+1} = \ga^k = \frac{1}{n}\one\tp\gF{\QQ^k} = \frac{1}{n}\one\tp\gF{\QQ^{k+1}}.
  }
  If $\xi^k = 1$, then $\QQ^{k+1} = \XX^k. $
  Again, by induction hypothesis and the fact $\one\tp\WW = \one\tp$, we have
  \eq{
    \ga^{k+1} &= \ga^k + \frac{\one\tp\pr{\gF{\XX^k} - \gF{\QQ^k}} }{n} =  \frac{\one\tp\gF{\QQ^k}}{n} + \frac{\one\tp\pr{\gF{\XX^k} - \gF{\QQ^k}} }{n}  \\
      &= \frac{1}{n}\one\tp\gF{\XX^k} = \frac{1}{n}\one\tp\gF{\QQ^{k+1}}.
  }
  Then, relation~\eqref{eq:gakm} follows by induction.

\subsection{Proof of Lemma~\ref{lem:dajrsupp}}\label{sec:prflemdajrsupp}
\begin{itemize}
    \item We first prove~\eqref{eq:dajrxz}.  \\
  The equality~\eqref{eq:xa} can be rewritten as follows:
  \eq{
    \xa^k - \za^k = \frac{1 - \alp - \tau}{\alp}\pr{\ya^k - \xa^k} + \frac{\tau}{\alp}\pr{\ua^k - \xa^k}.
  }
  By setting $\aa = \xa^k$ and $\bb = \ya^k$ in~\eqref{eq:inextL}, we have
  \eq{
    \jr{\da^k, \ya^k - \xa^k} \leq f\pr{\ya^k} - f\pr{\xa^k} + \frac{L}{n}\mt{\Con\XX^k}^2.
  }
%  Analogously, by applying~\eqref{eq:inextL} with $\aa = \xa^k$ and $\bb = \ua^k$, we have
%  \eq{
%    \jr{\da^k, \ua^k - \xa^k} \leq f\pr{\ua^k} - f\pr{\xa^k} + \frac{L}{n}\mt{\Con\XX^k}^2.
%  }
  Thus, by combining the above equations, we have
  \eq{
    &\jr{\da^k, \xa^k - \za^k} = \frac{1 - \alp - \tau}{\alp}\jr{\da^k, \ya^k - \xa^k} + \frac{\tau}{\alp}\jr{\da^k, \ua^k - \xa^k}  \\
    \leq& \frac{1 - \alp - \tau}{\alp}\pr{f\pr{\ya^k} - f\pr{\xa^k}} + \frac{\tau}{\alp}\jr{\da^k, \ua^k - \xa^k} + \frac{\pr{1 - \alp - \tau}L}{\alp n}\mt{\Con\XX^k}^2.
  }

    \item Next, we prove~\eqref{eq:dajrz+1}.  \\
  Rearranging~\eqref{eq:ya} yields
  \eql{\label{eq:jrz+1supp1}}{
    \za^{k+1} - \za^k = \frac{1}{\gam}\pr{\ya^{k+1} - \xa^k}.
  }
  By setting $\aa = \ya^{k+1}$ and $\bb = \xa^k$ in~\eqref{eq:inextL}, we have
  \eql{\label{eq:jrz+1supp2}}{
    f\pr{\ya^{k+1}} \leq f\pr{\xa^k} + \jr{\da^k, \ya^{k+1} - \xa^k} + \frac{L}{n}\mt{\Con\XX^k}^2 + L\nt{\xa^k - \ya^{k+1}}^2.
  }
  By~\eqref{eq:gakm}, the term $\nt{\da^k - \ga^k}$ can be bounded as follows:
  \eql{\label{eq:dagak}}{
    \nt{\da^k - \ga^k}^2 &= \frac{1}{n^2}\nt{\one\tp\pr{\gF{\XX^k} - \gF{\QQ^k}}}^2 \leq \frac{\nt{\one}^2}{n^2}\mt{\gF{\XX^k} - \gF{\QQ^k}}^2  \\
    &\leq \frac{1}{n}\mt{\gF{\XX^k} - \gF{\QQ^k}}^2.
  }
  Thus,
  \eq{
    &\Ek{\jr{\ga^k + \zeta^k\pr{\da^k - \ga^k}, \za^k - \za^{k+1}}}  \\
    =& \Ek{\jr{\da^k, \za^k - \za^{k+1}} + \jr{\pr{\zeta^k - 1}\pr{\da^k - \ga^k}, \za^k - \za^{k+1}}}  \\
    \comeq{\eqref{eq:jrz+1supp1}}& \Ek{\frac{1}{\gam}\jr{\da^k, \xa^{k}   - \ya^{k+1}} + \jr{\pr{\zeta^k - 1}\pr{\da^k - \ga^k}, \za^k - \za^{k+1}}}  \\
    \comleq{\eqref{eq:jrz+1supp2}}& \frac{1}{\gam}\Ek{f\pr{\xa^k} - f\pr{\ya^{k+1}}} + \frac{L}{n\gam}\mt{\Con\XX^k}^2 + \frac{L}{\gamma}\Ek{\nt{\xa^k - \ya^{k+1}}^2}  \\
     & + \frac{1}{4\eta}\Ek{\nt{\za^k - \za^{k+1}}^2} + 4\eta\Ek{\pr{\zeta^k - 1}^2}\nt{\da^k - \ga^k}^2  \\
    \comeq{\eqref{eq:jrz+1supp1}}& \frac{1}{\gam}\Ek{f\pr{\xa^k} - f\pr{\ya^{k+1}}} + \frac{L}{n\gam}\mt{\Con\XX^k}^2 + \pr{L\gam + \frac{1}{4\eta}}\Ek{\nt{\za^k - \za^{k+1}}^2}  \\
     & + 4\eta\pr{\frac{1}{q} - 1}\nt{\da^k - \ga^k}^2  \\
    \comleq{\eqref{eq:dagak}}& \frac{1}{\gam}\Ek{f\pr{\xa^k} - f\pr{\ya^{k+1}}} + \frac{L}{n\gam}\mt{\Con\XX^k}^2 + \pr{L\gam + \frac{1}{4\eta}}\Ek{\nt{\za^k - \za^{k+1}}^2}  \\
    & + \frac{4\eta}{qn}\mt{\gF{\XX^k} - \gF{\QQ^k}}^2.
  }

    \item Finally, we prove~\eqref{eq:dajrzaxo1}.  \\
  The equality~\eqref{eq:za} can be rewritten as
  \eq{
    \eta\pr{\ga^k + \zeta^k\pr{\da^k - \ga^k}} = {\za^k - \za^{k+1}} + \bet\pr{\xa^k - \za^{k+1}}.
  }
  Then, we have
  \eq{
    &2\eta\jr{\ga^k + \zeta^k\pr{\da^k - \ga^k}, \za^{k+1} - \xx^*}  \\
    =& 2\pr{\jr{\za^k - \za^{k+1}, \za^{k+1} - \xx^*} + \bet\jr{\xa^k - \za^{k+1}, \za^{k+1} - \xx^*}}  \\
    =& \nt{\za^k - \xx^*}^2 - \nt{\za^{k+1} - \xx^*}^2 - \nt{\za^k - \za^{k+1}}^2  \\
    & + \bet\pr{\nt{\xa^k - \xx^*}^2 - \nt{\za^{k+1} - \xx^*}^2 - \nt{\xa^k - \za^{k+1}}^2}  \\
    \leq& \nt{\za^k - \xx^*}^2 - \pr{1 + \bet}\nt{\za^{k+1} - \xx^*}^2 + \bet\nt{\xa^k - \xx^*}^2 - \nt{\za^k - \za^{k+1}}^2.
  }

  \end{itemize}

\subsection{Proof of Lemma~\ref{lem:caevol1}}\label{sec:prflemcaevol1}
    Consider matrix $\AA\in \MatSize{2n}{d}$ of the following form:
    \eq{
        \AA =
        \begin{pmatrix}
          \BB  \\
          \CC
        \end{pmatrix},
    }
    where $\BB, \CC \in \MatSize{n}{d}$ with $\frac{1}{n}\one\tp\BB = \frac{1}{n}\one\tp\CC$.
    Then,
    \eql{\label{eq:Wta1}}{
        \frac{1}{2n}\one\tp\Wt\AA = \frac{1}{2n}\pr{\pr{2 + \ew}\one\tp\BB - \ew\one\tp\CC} = \frac{1}{n}\one\tp\BB = \frac{1}{2n}\one\tp\AA,
    }
    and
    \eql{\label{eq:Wtaub}}{
        \frac{1}{n}\one\tp\br{\Wt\AA}_{1:n,:} = \frac{1 + \ew}{n}\one\tp\BB - \frac{\ew}{n}\one\tp\CC = \frac{1}{n}\one\tp\BB = \frac{1}{2n}\one\tp\AA.
    }

    Next, we show by induction that $\ua^k = \uat^k$.
    Firstly, since $\Ut^0 = \CX^0$, we have $\ua^0 = \uat^0$.
    Suppose we have shown $\ua^{k} = \uat^{k}$, we derive the same relationship for $k + 1$.

    When $\xi^k = 1$, we have $\Ut^{k+1} = \CX^k$.
    Therefore, by~\eqref{eq:Wta1} and~\eqref{eq:Wtaub}, we know
    \eq{
        \uat^{k+1} = \frac{1}{2n}\one\tp\Wt\CX^k = \xa^k
    }
    and
    \eq{
        \ua^{k+1} = \frac{1}{n}\one\tp\br{\Wt\CX^k}_{1:n,:} = \xa^k.
    }

    When $\xi^k = 0$, by the induction hypothesis,
    \eq{
        \Ut^k =
        \begin{pmatrix}
            &\UU^k &  \\
            &\qquad \quad \br{\Ut^k}_{(n+1):2n, :} & \quad
        \end{pmatrix}
    }
    satisfies $\frac{1}{n}\one\tp\br{\Ut^k}_{(n+1):2n,:} = 2\uat^k - \ua^k = \ua^k = \frac{1}{n}\one\tp\UU^k$.
    Hence we can apply~\eqref{eq:Wta1},~\eqref{eq:Wtaub} to $\Wt\Ut^k$ and obtain
    \eq{
        \uat^{k+1} = \frac{1}{2n}\one\tp\Wt\Ut^k = \ua^k
    }
    and
    \eq{
        \ua^{k+1} = \frac{1}{n}\one\tp\br{\Wt\Ut^k}_{1:n,:} = \ua^k.
    }
    Thus, in both cases, we obtain $\uat^{k+1} = \ua^{k+1}$.
    By induction, we have $\uat^k = \ua^k$ for any $k \geq 0$.
    We have also shown in the above analysis that
    \eql{\label{eq:cua}}{
        \ua^{k+1} = \pr{1 - \xi^k}\ua^k + \xi^k\xa^k.
    }

    From similar arguments as the above ones, we have $\za^k = \zat^k$ $(\forall k \geq 0)$ and
    \eq{
        \za^{k+1} = \zat^{k+1} = \pr{1 + \bet}\inv\pr{\za^k + \bet\xa^k - \eta\ga^k + \zeta^k\eta\pr{\ga^k - \da^k}}.
    }

    Again from similar arguments, we have $\gat^k = \ga^k$ for any $k \geq 0$.
    Then, by considering similar arguments for deriving Lemma~\ref{lem:gakm} and the fact that $\frac{1}{2n}\one\tp\cgF{\XX^k} = \frac{1}{n}\one\tp\gF{\XX^k}$, we have
    \eq{
        \ga^k = \frac{1}{n}\one\tp\gF{\QQ^k},\ \forall k\geq 0.
    }

    By multiplying $\frac{\one\tp}{n}$ on both sides of~\eqref{eq:CX}, ~\eqref{eq:CY}, ~\eqref{eq:CQ}, we have
    \eql{\label{eq:suppmc1}}{
        \xa^{k+1} = \pr{1 - \alp - \tau}\ya^k + \alp\za^k + \tau\ua^k,\ \ya^{k+1} = \xa^k + \gam\pr{\za^{k+1} - \za^k},\ \qa^{k+1} = \pr{1 - \xi^k}\qa^k + \xi^k\xa^k.
    }

    Since $\qa^0 = \ua^0 = \xa^0$, it can be shown by~\eqref{eq:cua},~\eqref{eq:suppmc1} and induction that
    \eq{
        \qa^k = \ua^k,\ \forall k \geq 0.
    }

\end{appendices}